\documentclass[12pt]{amsproc}
\usepackage{geometry} 
\usepackage{amsmath,amssymb,amsthm}
\usepackage[english]{babel}
\usepackage{tikz}
\usepackage{makecell}
\usepackage{graphicx}
\usepackage{setspace}
\usepackage{array}
\usepackage{comment}
\usepackage{multirow}
\usepackage{amsmath}
\usepackage{hyperref}
\hypersetup{hypertex=true,
colorlinks=true,
linkcolor=blue,
anchorcolor=blue,
citecolor=blue}
\allowdisplaybreaks[2]

\title{Sharp convergence rate of Schrödinger operator along curves}
\author{Meng Wang, Zhichao Wang}

\newtheorem{theorem}{Theorem}[section]
\newtheorem{lemma}[theorem]{Lemma}
\newtheorem{proposition}[theorem]{Proposition}

\begin{document}
	\begin{sloppypar}
		\maketitle
		\let\thefootnote\relax\footnotetext{This work is supported  in part by the National Natural Science Foundation of China (No.12371100 and No. 12171424)
			}
		\begin{abstract}
			We analyze convergence rate of Schrödinger operator along curves $U_{\gamma}f(x,t)$, where $ f \in H^{s}\left(\mathbb{R}^{d}\right) $. All results are sharp up to the endpoints.
		\end{abstract}
		
		\section{Introduction}
		\numberwithin{equation}{section}
		
		In this paper, we consider the Schrödinger equation as follows
		\[
			\left\{\begin{array}{ll}
				i u_{t}-\Delta u=0 & (x, t) \in \mathbb{R}^{d} \times \mathbb{R} \\
				u(x, 0)=f(x) & x \in \mathbb{R}^{d}.
			\end{array}\right.
		\]
		 The solution to this equation can be formally written as
		\[
			e^{i t (-\Delta)} f(x)=\frac{1}{(2\pi)^{d}}\int_{ \mathbb{R}^{d}} e^{i\left(x \cdot \xi+t|\xi|^{2}\right)} \widehat{f}(\xi) d \xi.
		\]
		Carleson \cite{MR0576038} put forward a question about exploring the minimal $s$ such that for any 
		$ f \in H^{s}\left(\mathbb{R}^{d}\right)$,
		there holds the almost everywhere convergence
		\begin{equation}\label{1.2}
		\underset{t \rightarrow 0}{\lim} ~e^{i t (-\Delta)} f(x)=f(x), \quad a.e.
		\end{equation}
		It has been studied by several authors \cite{MR0576038,MR0654188,MR0848141,MR0729347,MR0904948,MR0934859,MR1315543,MR1413873,MR1748920,MR2264734,MR3241836,MR3613507,Demeter2016SchrdingerMF,MR3574661,MR2044636,MR3903115,MR3702674,MR3961084}. In particular, Carleson \cite{MR0576038} pointed \( s \geq \frac{1}{4} \) when \( d=1 \) and  Du--Zhang \cite{MR3961084} proved \(s>\frac{d}{2(d+1)} \) when \( d\geq 2 \) are sharp results according to the counterexamples made by Dahlberg--Kenig \cite{MR0654188} in $d=1$ and
		Bourgain \cite{MR3574661}, Luc\`a--Rogers \cite{MR3903115} in $d\geq  2$.
		
		Some related issues have also been widely studied, including the pointwise convergence problem of the Schrödinger operator along curves
		\[
			U_\gamma f(x,t)=\frac{1}{(2\pi)^{d}}\int_{ \mathbb{R}^{d}}  e^{i\left(\gamma(x,t) \cdot \xi+t|\xi|^{2}\right)} \widehat{f}(\xi) d \xi.
		\]
		Here, we will always assume \(
		\gamma: \mathbb{R}^{d} \times[-1,1] \rightarrow \mathbb{R}^{d}, \quad \gamma(x, 0)=x
		\) is a continuous function satisfying bilipschitz continuous in \( x \), that is
		\[
		\frac{1}{C_{1}}\left|x-x^{\prime}\right| \leq\left|\gamma(x, t)-\gamma\left(x^{\prime}, t\right)\right| \leq C_{2}\left|x-x^{\prime}\right|, \quad t \in[-1,1], \quad x, x^{\prime} \in \mathbb{R}^{d} .
		\]
		The curve will be divided into tangential case $(0<\alpha<1)$ and non-tangential case $(\alpha=1)$ according to the Hölder continuous condition \( \alpha \in(0,1] \) in \( t \) , that is
		\[
			\left|\gamma(x, t)-\gamma\left(x, t^{\prime}\right)\right| \leq C_{3}\left|t-t^{\prime}\right|^{\alpha}, \quad x \in \mathbb{R}^{d}, \quad t, t^{\prime} \in[-1,1].
		\]
		
		For non-tangential case, Lee--Rogers \cite{MR2871144} showed the convergence problem is essentially equivalent to the traditional case (\ref{1.2}). The tangential case is more complex. In $d=1$, Cho--Lee--Vargas \cite{MR2970037} observed \( s>\max \left\{\frac{1}{4}, \frac{1-2 \alpha}{2}\right\} \) is the sharp sufficient condition. In $d\geq2$,  Minguill\'on \cite{MR4795804} and Cao--Miao \cite{MR4564456} got  the sharp result $s>\frac{d}{2(d+1)}$ when $\frac{1}{2}\leq \alpha\leq1$. However, the sharp sufficient condition for  $0< \alpha<\frac{1}{2}$ remains open.

		The convergence rate of Schrödinger operator is first considered by Cao-Fan-Wang \cite{MR3922421}, which has been improved to sharp by Wang--Zhao \cite{MR4877623} in $d=1$ and Fan--Wang \cite{Fan-Wang} in higher dimension.  To elucidate, after reaching the minimum regularity requirement $s>\frac{d}{2(d+1)}$ for almost everywhere convergence, if continues to increase the regularity of the function, what kind of convergence rate can it bring? It can be transformed to determine the optimal $s$ for \( f \in H^{s}\left(\mathbb{R}^{d}\right) \) to ensure
		\[
			\lim _{t \rightarrow 0} \frac{e^{i t (-\Delta)} f(x)-f(x)}{t^\delta}=0 \quad \text { a.e. } x \in \mathbb{R}^{d}.
		\]
		
		The convergence rate of Schrödinger operators along tangential curves is considered by Li--Wang \cite{MR4297035,Li2021OnCP}. In this paper, we investigate the convergence rate of the Schrödinger operator along curves $U_\gamma f(x,t)$, which is sharp to the endpoint.

	%
		
		\subsection{Convergence rate of Schrödinger operator along curves}\label{chap2}
		\numberwithin{equation}{section}
		
		~\\
		
		When the curve is bilipschitz continuous in \( x \) and Lipschitz continuous in \( t \), we obtain the following theorem. It's consistent with the properties of $e^{it(-\Delta)} f$, which is also intuitive.
		\begin{theorem}\label{theorem1}
			Let $\delta \in[0,1)$. For any curve $\gamma: \mathbb{R}^{d} \times[-1,1] \rightarrow \mathbb{R}^{d} $satisfying $\gamma(x,0)=x$  which is bilipschitz continuous in \( x \) and lipschitz continuous in \( t \), we have
			\begin{equation}\label{pointwise convergence in rn}
				\lim _{t \rightarrow 0} \frac{U_{\gamma} f(x, t)-f(x)}{t^\delta}=0 \quad \text { a.e. } x \in \mathbb{R}^{d} ,\quad \forall f \in H^{s}\left(\mathbb{R}^{d}\right)
			\end{equation}
			whenever 
			\[
			s>s(\delta)=\begin{cases}
				\frac{d}{2(d+1)} +\delta   \quad  &0 \leq \delta\leq \frac{d}{2(d+1)}, \\ 
				2\delta   \quad   &  \delta>\frac{d}{2(d+1)}.  \\
			\end{cases}
			\]
			Conversely, there exists a curve satisfying the previous assumption, while (\ref{pointwise convergence in rn}) fails whenever $s<s(\delta)$.
		\end{theorem}
		
		When $d=1$, we can discuss the case where the curve is $\alpha$-Hölder continuous in \( t \). We have the following three theorems.
		
		\begin{theorem}\label{theorem2}
			Let $\delta \in[0,\alpha)$. For any curve $\gamma(x,t): \mathbb{R} \times[-1,1] \rightarrow \mathbb{R}$ satisfying $\gamma(x,0)=x$ which   is bilipschitz continuous in \( x \), and $\alpha$-Hölder continuous in \( t \) with \(  \frac{1}{2}\leq \alpha < 1\), we have
			\begin{equation}\label{pointwise convergence in r1}
				\lim _{t \rightarrow 0} \frac{U_{\gamma} f(x, t)-f(x)}{t^\delta}=0 \quad \text { a.e. } x \in \mathbb{R},\quad \forall f \in H^{s}\left(\mathbb{R}\right)
			\end{equation}
			whenever 
			\[
			s>s(\delta)=\begin{cases}
				\frac{1}{4} +\delta   \quad  &0  \leq\delta\leq \frac{1}{4}, \\ 
				2\delta   \quad   &  \delta>\frac{1}{4}.  \\
			\end{cases}
			\]
			Conversely, there exists a curve satisfying the previous assumption, while (\ref{pointwise convergence in r1}) fails whenever $s<s(\delta)$.
			
		\end{theorem}
		
		\begin{theorem}\label{theorem3}
			Let $\delta \in[0,\alpha)$. For any curve $\gamma(x,t): \mathbb{R} \times[-1,1] \rightarrow \mathbb{R}$ satisfying $\gamma(x,0)=x$ which  is bilipschitz continuous in \( x \), and $\alpha$-Hölder continuous in \( t \) with \(  0 \leq \alpha \leq \frac{1}{4}  \), we have (\ref{pointwise convergence in r1}) whenever 
			\[
			s>s(\delta)=\begin{cases}
				\frac{1}{2}-\alpha +2\delta   \quad  &0  <\delta< \frac{\alpha}{2}, \\ 
				\frac{\delta}{\alpha}  \quad   &  \delta\geq\frac{\alpha}{2}.  \\
			\end{cases}
			\]
			Conversely, there exists a curve satisfying the previous assumption, while (\ref{pointwise convergence in r1}) fails whenever $s<s(\delta)$.
		\end{theorem}
		
		\begin{theorem}\label{theorem4}
			Let $\delta \in[0,\alpha)$. For any curve $\gamma(x,t): \mathbb{R} \times[-1,1] \rightarrow \mathbb{R}$ satisfying $\gamma(x,0)=x$ which is bilipschitz continuous in \( x \), and $\alpha$-Hölder continuous in \( t \) with \(  \frac{1}{4} < \alpha < \frac{1}{2}  \), we have (\ref{pointwise convergence in r1}) whenever
			\[
			s>s(\delta)=\begin{cases}
				\frac{1}{4} +\delta   \quad  &0 \leq \delta\leq \alpha-\frac{1}{4}, \\ 
				\frac{1}{2}-\alpha +2\delta   \quad  &\alpha-\frac{1}{4}  <\delta<\frac{\alpha}{2}, \\ 
				\frac{\delta}{\alpha}   \quad   &  \delta\geq \frac{\alpha}{2}.  \\
			\end{cases}
			\]
			Conversely, there exists a curve satisfying the previous assumption, while (\ref{pointwise convergence in r1}) fails whenever $s<s(\delta)$. \end{theorem}
		
		The following four graphs depict the relationship between the sharp convergence rate $\delta$ and the regularity $s$ of the Schrödinger operator along curves in four different scenarios: lipschitz continuous in \( t \) for $d\geq 1$, $\alpha$-Hölder continuous in \( t \) with \(  \frac{1}{2}\leq \alpha < 1\) for $d=1$, $\alpha$-Hölder continuous in \( t \) with \(  0 \leq \alpha \leq \frac{1}{4}  \)  for $d=1$ and $\alpha$-Hölder continuous in \( t \) with \(  \frac{1}{4} < \alpha < \frac{1}{2}  \) for $d=1$. More precise, if the pair $(s,\delta)$  is located in the area to the right of the line, endpoint not included, we can get	 (\ref{pointwise convergence in rn}).
		\begin{center}
					
			\tikzset{every picture/.style={line width=0.75pt}} 
			
			\begin{tikzpicture}[x=0.8pt,y=0.8pt,yscale=-1,xscale=1]
				
				\draw    (40,270) -- (254.42,268.56) ;
				\draw [shift={(256.42,268.55)}, rotate = 179.62] [color={rgb, 255:red, 0; green, 0; blue, 0 }  ][line width=0.75]    (10.93,-3.29) .. controls (6.95,-1.4) and (3.31,-0.3) .. (0,0) .. controls (3.31,0.3) and (6.95,1.4) .. (10.93,3.29)   ;
				\draw    (40,270) -- (39.01,136) ;
				\draw [shift={(39,134)}, rotate = 89.58] [color={rgb, 255:red, 0; green, 0; blue, 0 }  ][line width=0.75]    (10.93,-3.29) .. controls (6.95,-1.4) and (3.31,-0.3) .. (0,0) .. controls (3.31,0.3) and (6.95,1.4) .. (10.93,3.29)   ;
				\draw  [dash pattern={on 0.84pt off 2.51pt}]  (40,270) -- (148.05,195.91) ;
				\draw    (148.05,195.91) -- (199.03,161.75) ;
				\draw    (148.05,195.91) -- (102.33,270.17) ;
				\draw  [dash pattern={on 0.84pt off 2.51pt}]  (39.33,196.33) -- (148.05,195.91) ;
				\draw  [dash pattern={on 0.84pt off 2.51pt}]  (148.05,195.91) -- (149.71,268.98) ;
				\draw  [dash pattern={on 0.84pt off 2.51pt}]  (39,162) -- (199.03,161.75) ;
				\draw    (199.03,161.75) -- (247.7,161.56) ;
				\draw    (302,269) -- (516.42,267.56) ;
				\draw [shift={(518.42,267.55)}, rotate = 179.62] [color={rgb, 255:red, 0; green, 0; blue, 0 }  ][line width=0.75]    (10.93,-3.29) .. controls (6.95,-1.4) and (3.31,-0.3) .. (0,0) .. controls (3.31,0.3) and (6.95,1.4) .. (10.93,3.29)   ;
				\draw    (302,269) -- (301.01,131) ;
				\draw [shift={(301,129)}, rotate = 89.59] [color={rgb, 255:red, 0; green, 0; blue, 0 }  ][line width=0.75]    (10.93,-3.29) .. controls (6.95,-1.4) and (3.31,-0.3) .. (0,0) .. controls (3.31,0.3) and (6.95,1.4) .. (10.93,3.29)   ;
				\draw  [dash pattern={on 0.84pt off 2.51pt}]  (302,269) -- (410.05,194.91) ;
				\draw    (410.05,194.91) -- (461.03,160.75) ;
				\draw    (410.05,194.91) -- (364.33,269.17) ;
				\draw  [dash pattern={on 0.84pt off 2.51pt}]  (301.33,195.33) -- (410.05,194.91) ;
				\draw  [dash pattern={on 0.84pt off 2.51pt}]  (410.05,194.91) -- (411.71,267.98) ;
				\draw  [dash pattern={on 0.84pt off 2.51pt}]  (301,161) -- (461.03,160.75) ;
				\draw    (461.03,160.75) -- (509.7,160.56) ;
				\draw    (39.14,398.61) -- (526.39,396.72) ;
				\draw [shift={(528.39,396.72)}, rotate = 179.78] [color={rgb, 255:red, 0; green, 0; blue, 0 }  ][line width=0.75]    (10.93,-3.29) .. controls (6.95,-1.4) and (3.31,-0.3) .. (0,0) .. controls (3.31,0.3) and (6.95,1.4) .. (10.93,3.29)   ;
				\draw    (39.14,398.61) -- (40,303.59) ;
				\draw [shift={(40.02,301.59)}, rotate = 90.52] [color={rgb, 255:red, 0; green, 0; blue, 0 }  ][line width=0.75]    (10.93,-3.29) .. controls (6.95,-1.4) and (3.31,-0.3) .. (0,0) .. controls (3.31,0.3) and (6.95,1.4) .. (10.93,3.29)   ;
				\draw  [dash pattern={on 0.84pt off 2.51pt}]  (39.14,398.61) -- (287.55,346.28) ;
				\draw    (287.55,346.28) -- (405.8,327.02) ;
				\draw    (287.55,346.28) -- (180.05,398.83) ;
				\draw  [dash pattern={on 0.84pt off 2.51pt}]  (41.78,346.84) -- (287.55,346.28) ;
				\draw  [dash pattern={on 0.84pt off 2.51pt}]  (287.55,346.28) -- (287.14,397.28) ;
				\draw  [dash pattern={on 0.84pt off 2.51pt}]  (44.04,327.35) -- (405.8,327.02) ;
				\draw    (405.8,327.02) -- (515.83,326.77) ;
				\draw    (38.72,533.61) -- (525.96,531.72) ;
				\draw [shift={(527.96,531.72)}, rotate = 179.78] [color={rgb, 255:red, 0; green, 0; blue, 0 }  ][line width=0.75]    (10.93,-3.29) .. controls (6.95,-1.4) and (3.31,-0.3) .. (0,0) .. controls (3.31,0.3) and (6.95,1.4) .. (10.93,3.29)   ;
				\draw    (38.72,533.61) -- (39.57,438.59) ;
				\draw [shift={(39.59,436.59)}, rotate = 90.52] [color={rgb, 255:red, 0; green, 0; blue, 0 }  ][line width=0.75]    (10.93,-3.29) .. controls (6.95,-1.4) and (3.31,-0.3) .. (0,0) .. controls (3.31,0.3) and (6.95,1.4) .. (10.93,3.29)   ;
				\draw  [dash pattern={on 0.84pt off 2.51pt}]  (38.72,533.61) -- (287.12,481.28) ;
				\draw    (287.12,481.28) -- (405.38,462.02) ;
				\draw    (287.12,481.28) -- (240.34,504.71) ;
				\draw  [dash pattern={on 0.84pt off 2.51pt}]  (41.35,481.84) -- (287.12,481.28) ;
				\draw  [dash pattern={on 0.84pt off 2.51pt}]  (287.12,481.28) -- (286.72,532.28) ;
				\draw  [dash pattern={on 0.84pt off 2.51pt}]  (43.61,462.35) -- (405.38,462.02) ;
				\draw    (405.38,462.02) -- (515.4,461.77) ;
				\draw    (240.34,504.71) -- (217.34,533.71) ;
				\draw  [dash pattern={on 0.84pt off 2.51pt}]  (38.94,504.74) -- (240.34,504.71) ;
				\draw  [dash pattern={on 0.84pt off 2.51pt}]  (240.34,504.71) -- (241.06,532.97) ;
				\draw  [dash pattern={on 0.84pt off 2.51pt}]  (199.03,161.75) -- (199.43,269.19) ;
				\draw  [dash pattern={on 0.84pt off 2.51pt}]  (461.03,160.75) -- (461.43,268.19) ;
				\draw  [dash pattern={on 0.84pt off 2.51pt}]  (405.8,327.02) -- (406.09,396.95) ;
				\draw  [dash pattern={on 0.84pt off 2.51pt}]  (405.38,462.02) -- (405.66,531.95) ;
				
				\draw (25,128) node [anchor=north west][inner sep=0.75pt]   [align=left] {$\delta$};
				\draw (125,128) node [anchor=north west][inner sep=0.75pt]   [align=left] {$C_1$};
				\draw (355,128) node [anchor=north west][inner sep=0.75pt]   [align=left] {$C_\alpha\quad \alpha \in[\frac{1}{2},1)$};
				\draw (215,298) node [anchor=north west][inner sep=0.75pt]   [align=left] {$C_\alpha\quad \alpha \in(0,\frac{1}{4}]$};
				\draw (215,433) node [anchor=north west][inner sep=0.75pt]   [align=left] {$C_\alpha\quad \alpha \in(\frac{1}{4},\frac{1}{2})$};
				\draw (240.77,271.17) node [anchor=north west][inner sep=0.75pt]   [align=left] {$s$};
				\draw (75,272) node [anchor=north west][inner sep=0.75pt]   [align=left] {$\frac{d}{2(d+1)}$};
				\draw (130,272) node [anchor=north west][inner sep=0.75pt]   [align=left] {$\frac{d}{d+1}$};
				\draw (-5,186) node [anchor=north west][inner sep=0.75pt]   [align=left] {$\frac{d}{2(d+1)}$};
				\draw (288,123) node [anchor=north west][inner sep=0.75pt]   [align=left] {$\delta$};
				\draw (502.77,270.17) node [anchor=north west][inner sep=0.75pt]   [align=left] {$s$};
				\draw (360,271) node [anchor=north west][inner sep=0.75pt]   [align=left] {$\frac{1}{4}$};
				\draw (407,271) node [anchor=north west][inner sep=0.75pt]   [align=left] {$\frac{1}{2}$};
				\draw (288,185) node [anchor=north west][inner sep=0.75pt]   [align=left] {$\frac{1}{4}$};
				\draw (24,339.91) node [anchor=north west][inner sep=0.75pt]   [align=left] {$\frac{\alpha}{2}$};
				\draw (499.93,402.73) node [anchor=north west][inner sep=0.75pt]   [align=left] {$s$};
				\draw (177.19,403.8) node [anchor=north west][inner sep=0.75pt]   [align=left] {$\frac{1}{2}-\alpha$};
				\draw (283.44,403.8) node [anchor=north west][inner sep=0.75pt]   [align=left] {$\frac{1}{2}$};
				\draw (23,315.84) node [anchor=north west][inner sep=0.75pt]   [align=left] {$\alpha$};
				\draw (24,473.91) node [anchor=north west][inner sep=0.75pt]   [align=left] {$\frac{\alpha}{2}$};
				\draw (-3,495) node [anchor=north west][inner sep=0.75pt]   [align=left] {$\alpha-\frac{1}{4}$};
				\draw (499.5,537.73) node [anchor=north west][inner sep=0.75pt]   [align=left] {$s$};
				\draw (212.76,535) node [anchor=north west][inner sep=0.75pt]   [align=left] {$\frac{1}{4}$};
				\draw (283.01,535) node [anchor=north west][inner sep=0.75pt]   [align=left] {$\frac{1}{2}$};
				\draw (23,451.84) node [anchor=north west][inner sep=0.75pt]   [align=left] {$\alpha$};
				\draw (24,154) node [anchor=north west][inner sep=0.75pt]   [align=left] {$1$};
				\draw (287,154) node [anchor=north west][inner sep=0.75pt]   [align=left] {$\alpha$};
				\draw (194,272) node [anchor=north west][inner sep=0.75pt]   [align=left] {2};
				\draw (456,271) node [anchor=north west][inner sep=0.75pt]   [align=left] {$2\alpha$};
				\draw (23,297) node [anchor=north west][inner sep=0.75pt]   [align=left] {$\delta$};
				\draw (401,402) node [anchor=north west][inner sep=0.75pt]   [align=left] {${1}$};
				\draw (400,535) node [anchor=north west][inner sep=0.75pt]   [align=left] {${1}$};
				\draw (24,434) node [anchor=north west][inner sep=0.75pt]   [align=left] {$\delta$};
				\draw (235,538) node [anchor=north west][inner sep=0.75pt]   [align=left] {$\alpha$};

			\end{tikzpicture}
			
		\end{center}
		
		\subsection{Fractional cases}
		
		~\\
		
		We also consider the convergence rate of fractional Schrödinger operator along curves, that is, determining the optimal $s$ for \( f \in H^{s}\left(\mathbb{R}\right) \) to ensure 
		\[
		\lim _{t \rightarrow 0} \frac{U_{\gamma}^m f(x,t)-f(x)}{t^\delta}=0 \quad \text { a.e. } x \in \mathbb{R}.
		\]
		Here
		\[
		U_{\gamma}^m f(x,t):=\frac{1}{2\pi}\int_{ \mathbb{R}}  e^{i\left(\gamma(x,t) \cdot \xi+t|\xi|^{m}\right)} \widehat{f}(\xi) d \xi
		\]
		with $m>0$. The fractional Schrödinger operator
		\[
		e^{i t (-\Delta)^{\frac{m}{2}}} f(x)=\frac{1}{(2\pi)^{d}}\int_{ \mathbb{R}^{d}} e^{i\left(x \cdot \xi+t|\xi|^{m}\right)} \widehat{f}(\xi) d \xi
		\]
		also leads to find optimal $s$ for \( f \in H^{s}\left(\mathbb{R}^{d}\right) \) to get $\underset{t \rightarrow 0}{\lim} ~e^{i t (-\Delta)^{\frac{m}{2}}} f(x)=f(x)$ almost everywhere. In $d=1$, Sjölin \cite{MR0904948} and Walther \cite{MR1347033} showed that  $s\ge \frac{1}{4}$ when $m>1$ and $s> \frac{m}{4}$ when $0<m<1$ are sharp results. In higher dimension, Cowling \cite{MR0729347} and Zhang \cite{MR3247302}  pointed the convergence holds for $s>\{\frac{md}{4},\frac{m}{2}\}$ when $0<m<1$. Miao--Yang--Zheng \cite{MR3476484} and Cho--Ko \cite{MR4367790} obtained some results when $m>1$. The convergence problem of fractional Schrödinger operator along curves has also been studied. In $d=1$, Cho--Shiraki \cite{MR4307013} and Yuan-Zhao \cite{MR4359958} pointed $s>\max\{ \frac{1}{4},\frac{1-m\alpha}{2}\}$ when $m>1$ and $s>\max\{ \frac{2-m}{4},\frac{1-m\alpha}{2}\}$ when $0<m<1$ are sharp results.
		
		Since the proof of the convergence rate of fractional Schrödinger operator along curves is similar to the case $m=2$, we will only present results in  \textbf{\ref{chap5}} while omit the proof.\\
		
		\textbf{Outline.} In the following parts, we will prove the sufficiency in \textbf{\ref{chap3}} and necessity in \textbf{\ref{chap4}}. The convergence rate of fractional Schrödinger operator along curves is presented in \textbf{\ref{chap5}}.\\
		
		\textbf{Notation.} Throughout this article, we write \( A \gg B \) to mean if there is a large constant \( G \), which does not depend on \( A \) and \( B \) , such that \( A \geq G B \). We write $A \lesssim B$  to mean that there exists a constant $C$ such that $A \leq C B$. We write $A \gtrsim B$  to mean that there exists a constant $C$ such that $A \geq C B$. We write \( A \sim B \), and mean that \( A \) and \( B \) are comparable.  We write  supp \( \hat{f} \subset\left\{\xi:|\xi| \sim 2^k \right\} \) to mean  supp \( \hat{f} \subset\left\{\xi: \frac{2^k}{2}\le |\xi| \le 2^{k+1} \right\} \) and we will always assume $k \gg1$. We write $C_{X}$ to denote a constant that depends on $X$, where $X$ is a variable. We use $e_{1}$ to represent $(1,0, \ldots, 0)$. We write $A \lesssim_{\varepsilon} B$ to mean, there exists a constant $C_{\varepsilon}$ such that $A \leq C_{\varepsilon} B$.

		\section{Upper bounds for maximal functions}\label{chap3}
		\numberwithin{equation}{section}
		
		\textbf{Proof of Theorem 1.1}. We rewrite an important estimate from \cite{Wang-Wang} as a lemma.

		\begin{lemma}\label{proposition1}
			Let $J=(0,2^{-j})$ with $k\le j\le 2k$ is an interval. Suppose the curve $\gamma: \mathbb{R}^{d} \times[-1,1] \rightarrow \mathbb{R}^{d} $ satisfying $\gamma(x,0)=x$   is bilipschitz continuous in \( x \) and lipschitz continuous in \( t \). For any \( \varepsilon>0 \), we have
			\begin{equation}\label{3.2}
				\left\|\sup _{t \in J }\left|U_{\gamma} f\right|\right\|_{L^{2}(B(0,1))} \lesssim_{\varepsilon} 2^{\left(2k-j\right)\frac{d}{2(d+1)}+\varepsilon k}\|f\|_{L^{2}\left(\mathbb{R}^{d}\right)} 
			\end{equation}
			for all \( f \) with supp \( \hat{f} \subset\left\{\xi:|\xi|\sim 2^{k} \right\} \). 
		\end{lemma}
		
		\begin{proof}
			By a standard argument, we need to prove that the corresponding maximum estimate holds, that is for $0\leq \delta<1$, and arbitrary $\varepsilon>0$, we have
			\begin{equation}\label{standardest}
				\left\|\sup _{0<t<1} \frac{\left|U_\gamma f(x,t)-f(x)\right|}{t^{\delta}}\right\|_{L^{2}\left(B(0,1)\right)} \lesssim_{\varepsilon}\|f\|_{H^{s(\delta)+3\varepsilon}\left(\mathbb{R}^{d}\right)}.
			\end{equation}
			We can decompose \( f \) as
			\[
			f=\sum_{k=0}^{\infty} f_{k}
			\]
			where \( \operatorname{supp} \hat{f}_{0} \subset B(0,1), \operatorname{supp} \hat{f}_{k} \subset\left\{\xi:|\xi| \sim 2^{k}\right\}, k \geq 1 \). 
			By triangle inequality, it is obvious
			\[
			\left\|\sup _{0<t<1} \frac{\left|U_\gamma f(x,t)-f(x)\right|}{t^{\delta}}\right\|_{L^{2}\left(B(0,1)\right)} 
			\leq \sum_{k=0}^{\infty}\left\|\sup _{0<t<1} \frac{\left|U_\gamma f_k(x,t)-f_k(x)\right|}{t^{\delta}}\right\|_{L^{2}\left(B(0,1)\right)}.
			\]
			Using mean value theorem we can notice that
			\begin{equation}\label{klesssim1}
				\begin{aligned}
					\frac{\left|U_\gamma f_k(x,t)-f_k(\gamma(x,t))\right|}{t^{\delta }} &=\frac{1}{\left(2\pi \right)^{d}}\frac{\displaystyle \left| \int_{\mathbb{R}^{d}}e^{i\gamma(x,t)\cdot \xi }\left(e^{it|\xi|^2}-1\right)\hat{f_k}(\xi) d \xi\right|}{t^\delta}\\
					&\leq\frac{t^{1- \delta}}{\left(2\pi \right)^{d}}  \int_{\mathbb{R}^{d}}|\xi|^2|\hat{f_k}(\xi)| d \xi
				\end{aligned}
			\end{equation}
			and
			\begin{equation}\label{klesssim2}
				\begin{aligned}
					\frac{|f_k(\gamma(x, t))-f_k(x)|}{t^{\delta}} &=\frac{1}{\left(2\pi \right)^{d}}\frac{\displaystyle \left| \int_{\mathbb{R}^{d}}\left(e^{i\left(\gamma(x,t)-x\right)\cdot \xi }-1\right)e^{ix\cdot \xi}\hat{f_k}(\xi) d \xi\right|}{t^\delta}\\
					&\leq \frac{|\gamma(x, t)-x|}{t^{\delta}\left(2\pi \right)^{d}}\int_{\mathbb{R}^{d}}|\xi \|| \hat{f_k}(\xi)|d \xi \\
					&\leq \frac{t^{1- \delta}}{\left(2\pi \right)^{d}} \int_{\mathbb{R}^{d}}| \xi|| \hat{f_k}(\xi)| d \xi .
				\end{aligned}
			\end{equation}
			Thus for \( k \lesssim 1 \), one can get
			\begin{equation}\label{klesssim3}
				\begin{aligned}
					&\left\|\sup _{0<t<1} \frac{\left|U_\gamma f_k(x,t)-f_k(x)\right|}{t^{\delta}}\right\|_{L^{2}\left(B(0,1)\right)} \\
					&\leq\left\|\sup _{0<t<1} \frac{\left|U_\gamma f_k(x,t)-f_{k}(\gamma(x, t))\right|}{t^{\delta}}\right\|_{L^{2}\left(B(0,1)\right)} +\left\|\sup _{0<t<1} \frac{\left|f_{k}(\gamma(x, t))-f_{k}(x)\right|}{t^{\delta}}\right\|_{L^{2}\left(B(0,1)\right)} \\
					&\leq\left\|    \frac{1}{\left(2\pi \right)^{d}}  \int_{\mathbb{R}^{d}}|\xi|^2|\hat{f_k}(\xi)| d \xi    \right\|_{L^{2}\left(B(0,1)\right)} +\left\|\frac{1}{\left(2\pi \right)^{d}} \int_{\mathbb{R}^{d}}| \xi|| \hat{f_k}(\xi)| d \xi \right\|_{L^{2}\left(B(0,1)\right)} \\
					&\lesssim\|f\|_{H^{s(\delta)+3\varepsilon}\left(\mathbb{R}^{d}\right)}.
				\end{aligned}
			\end{equation}
			For \( k \gg 1 \), we will divide $t\in(0,1)$ into three parts, that is
			\begin{equation}\label{tdivision1}
				\begin{aligned}
					&\left\|\sup _{0<t<1} \frac{\left|U_\gamma f_k(x,t)-f_k(x)\right|}{t^{\delta}}\right\|_{L^{2}\left(B(0,1)\right)} \le \left\|\sup _{0<t<2^{-2k}} \frac{\left|U_\gamma f_k(x,t)-f_k(x)\right|}{t^{\delta}}\right\|_{L^{2}\left(B(0,1)\right)}+\\
					&\left\|\sup _{2^{-2k}\leq t< 2^{-k}} \frac{\left|U_\gamma f_k(x,t)-f_k(x)\right|}{t^{\delta}}\right\|_{L^{2}\left(B(0,1)\right)}+\left\|\sup _{2^{-k}\leq t<1} \frac{\left|U_\gamma f_k(x,t)-f_k(x)\right|}{t^{\delta}}\right\|_{L^{2}\left(B(0,1)\right)}\\
					&:=I+I I +I I I.
				\end{aligned}
			\end{equation}
			We first consider the first and third term estimates, which are relatively simple. For the first term, we have
			\begin{equation}\label{Iest}
				I \leq \left\|\sup _{0<t<2^{-2k}} \frac{\left|U_\gamma f_k(x,t)-f_k(\gamma(x, t))\right|}{t^{\delta}}\right\|_{L^{2}\left(B(0,1)\right)}+\left\|\sup _{0<t<2^{-2k}} \frac{\left|f_{k}(\gamma(x, t))-f_{k}(x)\right|}{t^{ \delta}}\right\|_{L^{2}\left(B(0,1)\right)}.
			\end{equation}
			We can estimate these two terms separately. Firstly Taylor's formula and the method of Lemma 5.1 in \cite{MR4297035} show that
			\begin{align*}
				&\left\|\sup _{0<t<2^{-2k}} \frac{\left|U_\gamma f_k(x,t)-f_k(\gamma(x, t))\right|}{t^{\delta}}\right\|_{L^{2}\left(B(0,1)\right)} \\
				&\leq \sum_{j=1}^{\infty} \frac{2^{-{2kj}+2{\delta k}}}{(2\pi)^d j !}\left\|\sup _{0<t<2^{-2k}}\left|\int_{\mathbb{R}^{d}} e^{i \gamma(x, t) \cdot \xi} |\xi|^{2j} \hat{f}_{k}(\xi) d \xi\right|\right\|_{L^{2}\left(B(0,1)\right)} \\
				&\leq \sum_{j=1}^{\infty} \frac{2^{-2{kj}+2{\delta k}}}{(2\pi)^d j !}\sum_{\mathfrak{l} \in \mathbb{Z}^{d}} \frac{C_d}{(1+|\mathfrak{l}|)^{d+1}}\left\|\int_{\mathbb{R}^{d}} e^{i\left(x+\frac{\mathfrak{l}}{2^{k}}\right) \cdot \xi} |\xi|^{2j} \hat{f}_{k}(\xi) d \xi\right\|_{L^{2}\left(B(0,1)\right)} \\
				&\leq \sum_{j=1}^{\infty}\frac{2^{-2{kj}+2{\delta k}}}{(2\pi)^dj !}\sum_{\mathfrak{l} \in \mathbb{Z}^d} \frac{C_d}{(1+|\mathfrak{l}|)^{d+1}}\left\||\xi|^{2j} \hat{f}_{k}(\xi)\right\|_{L^{2}\left(\mathbb{R}^{d}\right)} \\
				&\lesssim \sum_{j=1}^{\infty} \frac{2^{-2{kj}+2{\delta k}} 2^{2 k j}}{(2\pi)^d j !}\left\|\hat{f}_{k}(\xi)\right\|_{L^{2}\left(\mathbb{R}^{d}\right)} \\
				&\lesssim 2^{2{\delta k}}\|{f}_{k}\|_{L^{2}\left(\mathbb{R}^{d}\right)}.
			\end{align*}
			For the next term, we can handle it via Taylor's formula only, that is
			\begin{align*}
				&\left\|\sup _{0<t<2^{-2k}} \frac{\left|f_{k}(\gamma(x, t))-f_{k}(x)\right|}{t^{ \delta}}\right\|_{L^{2}\left(B(0,1)\right)} \\
				&\leq \sum_{j \geq 1} \frac{1}{(2\pi)^d j !} \sum_{h_1,h_2,…,h_j \in\left\{1,2,…, d\right\} }\left\|\sup_{0<t<2^{-2k}} \frac{\prod \limits_{d=1}^{j}\left|\gamma_{h_d}(x, t)-x_{h_d}\right|}{t^{ \delta}}\left|\int_{\mathbb{R}} e^{i x \cdot \xi}  \prod \limits_{d=1}^{j} \xi_{h_d} \hat{f}_{k}(\xi) d \xi\right|\right\|_{L^{2}\left(B(0,1)\right)} \\
				&\leq \sum_{j \geq 1} \frac{2^{-2k j+2k\delta}}{(2\pi)^d j !} \sum_{h_1,h_2,…,h_j \in\left\{1,2,…, d\right\} }\left\|\int_{\mathbb{R}} e^{i x \cdot \xi}   \prod \limits_{d=1}^{j} \xi_{h_d} \hat{f}_{k}(\xi) d \xi\right\|_{L^{2}\left(B(0,1)\right)} \\
				&\lesssim \sum_{j \geq 1} \frac{2^{-2k j+2k\delta} 2^{k j} }{(2\pi)^d j !}\left\|\hat{f}_{k}\right\|_{L^{2}\left(\mathbb{R}^{d}\right)} \\
				&\lesssim 2^{2k\delta}\left\|{f}_{k}\right\|_{L^{2}\left(\mathbb{R}^{d}\right)},
			\end{align*}
			where $\gamma(x,t)=\left(\gamma_1(x,t),\gamma_2(x,t),…,\gamma_d(x,t)\right)$. As for the third term, it is obvious that
			\begin{equation}\label{divison3}
				\begin{aligned}
					\left\|\sup _{2^{-k}\leq t<1} \frac{\left|U_\gamma f_k(x,t)-f_k(x)\right|}{t^{\delta}}\right\|_{L^{2}\left(B(0,1)\right)}&\leq 2^{k\delta}\left\|\sup _{2^{-k}\leq t<1} \left|U_\gamma f_k(x,t)\right|\right\|_{L^{2}\left(B(0,1)\right)}+2^{k\delta}\|{f}_{k}\|_{L^{2}\left(\mathbb{R}^{d}\right)}\\
					&\leq 2^{k\delta}\left\|\sup _{t\in(0,1)} \left|U_\gamma f_k(x,t)\right|\right\|_{L^{2}\left(B(0,1)\right)}+2^{k\delta}\|{f}_{k}\|_{L^{2}\left(\mathbb{R}^{d}\right)}\\
					&\lesssim 2^{k\left(\delta+\frac{d}{2(d+1)}\right)}\|{f}_{k}\|_{L^{2}\left(\mathbb{R}^{d}\right)}.
				\end{aligned}
			\end{equation}
			Next we will consider $t\in (2^{-2k},2^{-k})$, dyadic decomposition of time shows that
			\begin{align*}
				&\left\|\sup _{2^{-2k}\leq t< 2^{-k}} \frac{\left|U_\gamma f_k(x,t)-f_k(x)\right|}{t^{\delta}}\right\|_{L^{2}\left(B(0,1)\right)} \\
				&\leq  \left\|\sup _{2^{-2k}\leq t< 2^{-k}} \frac{\left|U_\gamma f_k(x,t)\right|}{t^{\delta}}\right\|_{L^{2}\left(B(0,1)\right)} +\left\|\sup _{2^{-2k}\leq t< 2^{-k}} \frac{\left|f_k(x)\right|}{t^{\delta}}\right\|_{L^{2}\left(B(0,1)\right)}\\ 
				&\leq \left\|\sup _{k\leq j \leq2k} \sup _{2^{-j-1}\leq t< 2^{-j}} \frac{\left|U_\gamma f_k(x,t)\right|}{t^{\delta}}\right\|_{L^{2}\left(B(0,1)\right)}+ 2^{2{\delta k}}\|{f}_{k}\|_{L^{2}\left(\mathbb{R}^{d}\right)}\\
				&\leq 2\sum _{k\leq j \leq 2k}2^{j\delta} \left\| \sup _{2^{-j-1}\leq t< 2^{-j}} \left|U_\gamma f_k(x,t)\right|\right\|_{L^{2}\left(B(0,1)\right)}+ 2^{2{\delta k}}\|{f}_{k}\|_{L^{2}\left(\mathbb{R}^{d}\right)}\\
				&\lesssim  \sum _{k\leq j \leq 2k}2^{j\delta} \left\| \sup _{t\in (0,2^{-j})} \left|U_\gamma f_k(x,t)\right|\right\|_{L^{2}\left(B(0,1)\right)}
				+2^{2{\delta k}}\|{f}_{k}\|_{L^{2}\left(\mathbb{R}^{d}\right)}.
			\end{align*}
			Lemma \ref{proposition1} shows that
			\begin{align*}
				I I&\lesssim_{\varepsilon }    \left( \sum _{k\leq j \leq 2k}2^{j\left(\delta-\frac{d}{2(d+1)}\right)+\left(\frac{d}{d+1}+\varepsilon        \right)k}+2^{2{\delta k}}\right)\|{f}_{k}\|_{L^{2}\left(\mathbb{R}^{d}\right)}\\
				&\lesssim_{\varepsilon}
				\begin{cases}
					k2^{k\left(\delta+\frac{d}{2(d+1)}+\varepsilon \right)}\|{f}_{k}\|_{L^{2}\left(\mathbb{R}^{d}\right)}&\delta \in\left[0,\frac{d}{2(d+1)}\right),\\
					k2^{k\left(2{\delta}+\varepsilon\right)}\|{f}_{k}\|_{L^{2}\left(\mathbb{R}^{d}\right)}&\delta \in\left[\frac{d}{2(d+1)},1 \right).
				\end{cases}
			\end{align*}
			Notice that  $k \lesssim_{\varepsilon} 2^{\varepsilon   k }$ for $k\gg1$. Hence $I+I I+I I I$ can be estimated by
			\[
			\left\|\sup _{0<t<1} \frac{\left|U_\gamma f_k(x,t)-f_k(x)\right|}{t^{\delta}}\right\|_{L^{2}\left(B(0,1)\right)} \lesssim_{\varepsilon}
			\begin{cases}
				2^{k\left(\delta+\frac{d}{2(d+1)}+2\varepsilon\right)}\|{f}_{k}\|_{L^{2}\left(\mathbb{R}^{d}\right)}&\delta \in\left[0,\frac{d}{2(d+1)}\right),\\
				2^{k\left(2{\delta}+2\varepsilon\right)}\|{f}_{k}\|_{L^{2}\left(\mathbb{R}^{d}\right)}&\delta \in\left[\frac{d}{2(d+1)},1 \right).
			\end{cases}
			\]
			We can choose  $s_1=s(\delta)+3\varepsilon$, then we have
			\begin{align*}
				&\left\|\sup _{0<t<1} \frac{\left|U_\gamma f(x,t)-f(x)\right|}{t^{\delta}}\right\|_{L^{2}\left(B(0,1)\right)} \\
				&\leq \sum_{k\lesssim 1}\left\|\sup _{0<t<1} \frac{\left|U_\gamma f_k(x,t)-f_k(x)\right|}{t^{\delta}}\right\|_{L^{2}\left(B(0,1)\right)}+
				\sum_{k\gg 1}\left\|\sup _{0<t<1} \frac{\left|U_\gamma f_k(x,t)-f_k(x)\right|}{t^{\delta}}\right\|_{L^{2}\left(B(0,1)\right)} \\
				&\lesssim_{\varepsilon} \|f\|_{H^{s_{1}}}+\sum_{k=1}^{\infty}2^{k(s(\delta)+2\varepsilon)}\|f_k\|_2\\
				&\lesssim_{\varepsilon} \|f\|_{H^{s_1}},
			\end{align*}
			which completes the proof of (\ref{standardest}).
			
		\end{proof}

		\textbf{Proof of Theorem \ref{theorem2}}. We rewrite an important estimate from \cite{Wang-Wang} as a lemma.
		
		\begin{lemma}\label{proposition2}
			Let $J=(0,2^{-j})$ with $k\le j\le 2k$ is an interval. For any curve $\gamma(x,t): \mathbb{R} \times[-1,1] \rightarrow \mathbb{R}$ satisfying $\gamma(x,0)=x$ which is bilipschitz continuous in \( x \), and $\alpha$-Hölder continuous in \( t \) with \(  \frac{1}{2}\leq \alpha <1\), we have
			\begin{equation}\label{0.5-1max}
				\left\|\sup _{t \in J}\left|U_\gamma f(\cdot, t)\right|\right\|_{L^{2}([-1,1])}\lesssim 2^{\frac{1}{4}\left(2k-j\right)} \|f\|_{2}   \end{equation}
			for all \( f \) with supp \( \hat{f} \subset\left\{\xi:|\xi|\sim 2^k \right\} \). 
		\end{lemma}
		\begin{proof}
			As the previous proof, we can decompose \( f \) as $f=\sum_{k=0}^{\infty} f_{k}$. If $k\lesssim 1$, the proof method of (\ref{klesssim1}) and (\ref{klesssim2})  can be equally applicable to get (\ref{klesssim3}) as well. If $k\gg 1$, we can divide $t\in(0,1)$ into three parts as
			(\ref{tdivision1}) to define 
			\[
			\left\|\sup _{0<t<1} \frac{\left|U_\gamma f_k(x,t)-f_k(x)\right|}{t^{\delta}}\right\|_{L^{2}([-1,1])}:=I +I I+ I I I.
			\]
			For the first term $I$, we will treat it as (\ref{Iest}) and obtain the same estimate for the first term, but slightly different for the second term, which is
			\begin{align*}
				&\left\|\sup _{0<t<2^{-2k}} \frac{\left|f_{k}(\gamma(x, t))-f_{k}(x)\right|}{t^{ \delta}}\right\|_{L^{2}([-1,1])} \\
				&\lesssim \sum_{j \geq 1} \frac{1}{(2\pi)j !} \left\|\sup_{0<t<2^{-2k}} \frac{\left|\gamma(x, t)-x\right|^j}{t^{ \delta}}\left|\int_{\mathbb{R}} e^{i x \cdot \xi}  \xi^j \hat{f}_{k}(\xi) d \xi\right|\right\|_{L^{2}([-1,1])} \\
				&\lesssim \sum_{j \geq 1} \frac{2^{-2\alpha k j+2k\delta}}{(2\pi)j !} \left\|\int_{\mathbb{R}} e^{i x \cdot \xi}  \xi^j\hat{f}_{k}(\xi) d \xi\right\|_{L^{2}([-1,1])} \\
				&\lesssim \sum_{j \geq 1} \frac{2^{-2\alpha k j+2k\delta} 2^{k j} }{(2\pi)j !}\left\|\hat{f}_{k}\right\|_{L^{2}\left(\mathbb{R}\right)} \\
				&\lesssim 2^{2k\delta}\left\|{f}_{k}\right\|_{L^{2}\left(\mathbb{R}\right)}. \\
			\end{align*}
			As for the third term, we can handle it similar to (\ref{divison3}), that is
			\[
			I I I\lesssim 2^{k\left(\delta+\frac{1}{4}\right)}\|{f}_{k}\|_{L^{2}\left(\mathbb{R}\right)}.
			\]
			As for the second term, we can get
			\begin{align*}
				&\left\|\sup _{2^{-2k}\leq t< 2^{-k}} \frac{\left|U_\gamma f_k(x,t)-f_k(x)\right|}{t^{\delta}}\right\|_{L^{2}([-1,1])} \\
				&\leq \left\|\sup _{k\leq j \leq2k} \sup _{2^{-j-1}\leq t< 2^{-j}} \frac{\left|U_\gamma f_k(x,t)\right|}{t^{\delta}}\right\|_{L^{2}([-1,1])}+ 2^{2{\delta k}}\|{f}_{k}\|_{L^{2}\left(\mathbb{R}\right)}\\
				&\leq 2\sum _{k\leq j \leq 2k}2^{j\delta} \left\| \sup _{2^{-j-1}\leq t< 2^{-j}} \left|U_\gamma f_k(x,t)\right|\right\|_{L^{2}([-1,1])}+ 2^{2{\delta k}}\|{f}_{k}\|_{L^{2}\left(\mathbb{R}\right)}\\
				&\lesssim  \sum _{k\leq j \leq 2k}2^{j\delta} \left\| \sup _{t\in (0,2^{-j})} \left|U_\gamma f_k(x,t)\right|\right\|_{L^{2}([-1,1])}+2^{2{\delta k}}\|{f}_{k}\|_{L^{2}\left(\mathbb{R}\right)}.\\
			\end{align*}
			Lemma \ref{proposition2} shows that
			\begin{align*}
				I I&\lesssim\left( \sum _{k\leq j \leq 2k}2^{j\left(\delta-\frac{1}{4}\right)+\frac{k}{2}}+2^{2{\delta k}}\right)\|{f}_{k}\|_{L^{2}\left(\mathbb{R}\right)}\\
				&\lesssim
				\begin{cases}
					k2^{k\left(\delta+\frac{1}{4}\right)}\|{f}_{k}\|_{L^{2}\left(\mathbb{R}\right)}&\delta \in\left[0,\frac{1}{4}\right),\\
					k2^{2k{\delta}}\|{f}_{k}\|_{L^{2}\left(\mathbb{R}\right)}&\delta \in\left[\frac{1}{4},\alpha \right),\\
				\end{cases}
			\end{align*}
			Hence, for any $\varepsilon>0 $,  $I+I I+I I I$ can be estimated by
			\[
			\left\|\sup _{0<t<1} \frac{\left|U_\gamma f_k(x,t)-f_k(x)\right|}{t^{\delta}}\right\|_{L^{2}([-1,1])} \lesssim_{\varepsilon}
			\begin{cases}
				2^{k\left(\delta+\frac{1}{4}+\varepsilon\right)}\|{f}_{k}\|_{L^{2}\left(\mathbb{R}\right)}&\delta \in\left[0,\frac{1}{4}\right),\\
				2^{k\left(2{\delta}+\varepsilon\right)}\|{f}_{k}\|_{L^{2}\left(\mathbb{R}\right)}&\delta \in\left[\frac{1}{4},\alpha \right).
			\end{cases}
			\]
			
		\end{proof}

		\textbf{Proof of Theorem \ref{theorem3}}. We rewrite an important estimate from \cite{Wang-Wang} as a lemma.
		\begin{lemma}\label{proposition3}
			Let $J=(0,2^{-j})$ with $2k\le j\le \frac{k}{\alpha}$ is an interval. For any curve $\gamma(x,t): \mathbb{R} \times[-1,1] \rightarrow \mathbb{R}$ satisfying $\gamma(x,0)=x$ is bilipschitz continuous in \( x \), and $\alpha$-Hölder continuous in \( t \) with \(  0 < \alpha \leq \frac{1}{4}\), we have
			
			\begin{equation}\label{0-0.25max}
				\left\|\sup _{t \in J}\left|U_\gamma f(\cdot, t)\right|\right\|_{L^{2}([-1,1])} \lesssim 2^{\frac{1}{2}\left(k-\alpha j \right)}\|f\|_{2} 
			\end{equation}
			
			for all \( f \) with supp\( \hat{f} \subset\left\{\xi:|\xi|\sim 2^k \right\} \).
		\end{lemma}
		
		\begin{proof}
			As the previous proof, we can decompose \( f \) as $f=\sum_{k=0}^{\infty} f_{k}$. If $k\lesssim 1$, we can get (\ref{klesssim3}) as well. If $k\gg 1$, we can  divide $t\in(0,1)$ into three parts, that is
			\begin{equation}\label{divisonalpha}
				\begin{aligned}
					&\left\|\sup _{0<t<1} \frac{\left|U_\gamma f_k(x,t)-f_k(x)\right|}{t^{\delta}}\right\|_{L^{2}([-1,1])} \le \left\|\sup _{0<t<2^{-\frac{k}{\alpha}}} \frac{\left|U_\gamma f_k(x,t)-f_k(x)\right|}{t^{\delta}}\right\|_{L^{2}([-1,1])}+\\
					&\left\|\sup _{2^{-\frac{k}{\alpha}}\leq t< 2^{-2k}} \frac{\left|U_\gamma f_k(x,t)-f_k(x)\right|}{t^{\delta}}\right\|_{L^{2}([-1,1])}+\left\|\sup _{2^{-2k}\leq t<1} \frac{\left|U_\gamma f_k(x,t)-f_k(x)\right|}{t^{\delta}}\right\|_{L^{2}([-1,1])}\\
					&:=I+I I +I I I.
				\end{aligned}
			\end{equation}
			
			For the first term, notice that
			\begin{equation}\label{firstterm2-1}
				I \leq \left\|\sup _{0<t<2^{-\frac{k}{\alpha}}} \frac{\left|U_\gamma f_k(x,t)-f_k(\gamma(x, t))\right|}{t^{\delta}}\right\|_{L^{2}([-1,1])}+\left\|\sup _{0<t<2^{-\frac{k}{\alpha}}} \frac{\left|f_{k}(\gamma(x, t))-f_{k}(x)\right|}{t^{ \delta}}\right\|_{L^{2}([-1,1])}
			\end{equation}
			We can estimate these two terms separately. Firstly, Taylor's formula and the method of Lemma 5.1 in \cite{MR4297035} show that
			\begin{equation}\label{firstterm2-2}
				\begin{aligned}
					&\left\|\sup _{0<t<2^{-\frac{k}{\alpha}}} \frac{\left|U_\gamma f_k(x,t)-f_k(\gamma(x, t))\right|}{t^{\delta}}\right\|_{L^{2}([-1,1])} \\
					&\leq \sum_{j=1}^{\infty} \frac{2^{-\frac{kj}{\alpha}+\frac{\delta k}{\alpha}}}{(2\pi)j !}\left\|\sup _{0<t<2^{-\frac{k}{\alpha}}}\left|\int_{\mathbb{R}} e^{i \gamma(x, t) \cdot \xi} |\xi|^{2j} \hat{f}_{k}(\xi) d \xi\right|\right\|_{L^{2}([-1,1])} \\
					&\leq \sum_{j=1}^{\infty} \frac{2^{-\frac{kj}{\alpha}+\frac{\delta k}{\alpha}}}{(2\pi)j !}\sum_{\mathfrak{l} \in \mathbb{Z}} \frac{C}{(1+|\mathfrak{l}|)^{2}}\left\|\int_{\mathbb{R}} e^{i\left(x+\frac{\mathfrak{l}}{2^{k}}\right) \cdot \xi} |\xi|^{2j} \hat{f}_{k}(\xi) d \xi\right\|_{L^{2}([-1,1])} \\
					&\leq \sum_{j=1}^{\infty}\frac{2^{-\frac{kj}{\alpha}+\frac{\delta k}{\alpha}}}{(2\pi)j !}\sum_{\mathfrak{l} \in \mathbb{Z}} \frac{C}{(1+|\mathfrak{l}|)^{2}}\left\||\xi|^{2j} \hat{f}_{k}(\xi)\right\|_{L^{2}\left(\mathbb{R}\right)} \\
					&\lesssim \sum_{j=1}^{\infty} \frac{2^{-\frac{kj}{\alpha}+\frac{\delta k}{\alpha}} 2^{2 k j}}{(2\pi)j !}\left\|\hat{f}_{k}(\xi)\right\|_{L^{2}\left(\mathbb{R}\right)} \\
					&\lesssim 2^{\frac{\delta k}{\alpha}}\|{f}_{k}\|_{L^{2}\left(\mathbb{R}\right)}
				\end{aligned}
			\end{equation}
			For the next term, we can handle it via Taylor's formula only, that is
			\begin{equation}\label{firstterm2-3}
				\begin{aligned}
					&\left\|\sup _{0<t<2^{-\frac{k}{\alpha}}} \frac{\left|f_{k}(\gamma(x, t))-f_{k}(x)\right|}{t^{ \delta}}\right\|_{L^{2}([-1,1])} \\
					&\lesssim \sum_{j \geq 1} \frac{1}{(2\pi)j !} \left\|\sup_{0<t<2^{-\frac{k}{\alpha}}} \frac{\left|\gamma(x, t)-x\right|^j}{t^{ \delta}}\left|\int_{\mathbb{R}} e^{i x \cdot \xi}  \xi^j \hat{f}_{k}(\xi) d \xi\right|\right\|_{L^{2}([-1,1])} \\
					&\lesssim \sum_{j \geq 1} \frac{2^{-k j+\frac{\delta k}{\alpha}}}{(2\pi)j !} \left\|\int_{\mathbb{R}} e^{i x \cdot \xi}  \xi^j\hat{f}_{k}(\xi) d \xi\right\|_{L^{2}([-1,1])} \\
					&\lesssim \sum_{j \geq 1} \frac{2^{-k j+\frac{\delta k}{\alpha}} 2^{k j} }{(2\pi)j !}\left\|\hat{f}_{k}\right\|_{L^{2}\left(\mathbb{R}\right)} \\
					&\lesssim 2^{\frac{\delta k}{\alpha}}\left\|{f}_{k}\right\|_{L^{2}\left(\mathbb{R}\right)}. \\
				\end{aligned}
			\end{equation}
			As for the third term, we can handle it similar to (\ref{divison3}), that is
			\[
			\left\|\sup _{2^{-2k}\leq t<1} \frac{\left|U_\gamma f_k(x,t)-f_k(x)\right|}{t^{\delta}}\right\|_{L^{2}([-1,1])}\lesssim 2^{k\left(2\delta+\frac{1}{2}-\alpha\right)}\|{f}_{k}\|_{L^{2}\left(\mathbb{R}\right)}
			\]
			As for the second term, we can get
			\begin{align*}
				&\left\|\sup _{2^{-\frac{k}{\alpha}}\leq t< 2^{-2k}} \frac{\left|U_\gamma f_k(x,t)-f_k(x)\right|}{t^{\delta}}\right\|_{L^{2}([-1,1])} \\
				&\leq \left\|\sup _{2k\leq j \leq \frac{k}{\alpha}} \sup _{2^{-j-1}\leq t< 2^{-j}} \frac{\left|U_\gamma f_k(x,t)\right|}{t^{\delta}}\right\|_{L^{2}([-1,1])}+ 2^{\frac{\delta k}{\alpha}}\|{f}_{k}\|_{L^{2}\left(\mathbb{R}^{d}\right)}\\
				&\leq 2\sum _{2k\leq j \leq \frac{k}{\alpha}}2^{j\delta} \left\| \sup _{2^{-j-1}\leq t< 2^{-j}} \left|U_\gamma f_k(x,t)\right|\right\|_{L^{2}([-1,1])}+ 2^{\frac{\delta k}{\alpha}}\|{f}_{k}\|_{L^{2}\left(\mathbb{R}\right)}\\
				&\lesssim  \sum _{2k\leq j \leq \frac{k}{\alpha}}2^{j\delta} \left\| \sup _{t\in (0,2^{-j})} \left|U_\gamma f_k(x,t)\right|\right\|_{L^{2}([-1,1])}+2^{\frac{\delta k}{\alpha}}\|{f}_{k}\|_{L^{2}\left(\mathbb{R}\right)}.
			\end{align*}
			Lemma \ref{proposition2} shows that
			\begin{align*}
				I I&\lesssim\left( \sum _{2k\leq j \leq \frac{k}{\alpha}}2^{j\left(\delta-\frac{\alpha}{2}\right)+\frac{k}{2}}+2^{\frac{\delta k}{\alpha}}\right)\|{f}_{k}\|_{L^{2}\left(\mathbb{R}\right)}\\
				&\lesssim
				\begin{cases}
					\left(k2^{k\left(2\delta+\frac{1}{2}-\alpha \right)}+2^{k\frac{\delta}{\alpha}}\right)\|{f}_{k}\|_{L^{2}\left(\mathbb{R}\right)}&\delta \in\left[0,\frac{\alpha}{2}\right),\\
					\left(k2^{k\frac{\delta}{\alpha}}+2^{k\frac{\delta}{\alpha}}\right)\|{f}_{k}\|_{L^{2}\left(\mathbb{R}\right)}&\delta \in\left[\frac{\alpha}{2},\alpha \right),\\
				\end{cases}\\
				&\lesssim
				\begin{cases}
					k2^{k\left(2\delta+\frac{1}{2}-\alpha\right)}\|{f}_{k}\|_{L^{2}\left(\mathbb{R}\right)}&\delta \in\left[0,\frac{\alpha}{2}\right),\\
					k2^{k\frac{\delta}{\alpha}}\|{f}_{k}\|_{L^{2}\left(\mathbb{R}\right)}&\delta \in\left[\frac{\alpha}{2},\alpha \right).
				\end{cases}
			\end{align*}
			Hence, for any $\varepsilon>0 $,  $I+I I+I I I$ can be estimated by
			\[
			\left\|\sup _{0<t<1} \frac{\left|U_\gamma f_k(x,t)-f_k(x)\right|}{t^{\delta}}\right\|_{L^{2}([-1,1])} \lesssim_{\varepsilon}
			\begin{cases}
				2^{k\left(2\delta+\frac{1}{2}-\alpha+\varepsilon\right)}\|{f}_{k}\|_{L^{2}\left(\mathbb{R}\right)}&\delta \in\left[0,\frac{\alpha}{2}\right),\\
				2^{k\left(\frac{\delta}{\alpha}+\varepsilon\right)}\|{f}_{k}\|_{L^{2}\left(\mathbb{R}\right)}&\delta \in\left[\frac{\alpha}{2},\alpha \right).\\
			\end{cases}
			\]			
		\end{proof}

		\textbf{Proof of Theorem \ref{theorem4}}. We rewrite an important estimate from \cite{Wang-Wang} as a lemma.
		
		\begin{lemma}\label{proposition4}
			Let $J=(0,2^{-j})$ with $k\le j\le \frac{k}{\alpha}$ is an interval. For any curve $\gamma(x,t): \mathbb{R} \times[-1,1] \rightarrow \mathbb{R}$ satisfying $\gamma(x,0)=x$ is bilipschitz continuous in \( x \), and $\alpha$-Hölder continuous in \( t \) with \(  \frac{1}{4}< \alpha < \frac{1}{2}\), we have
			\begin{equation}\label{0.25-0.5max}
				\left\|\sup _{t \in J}\left|U_{\gamma} f(\cdot, t)\right|\right\|_{L^{2}([-1,1])}\lesssim 
				\begin{cases}
					2^{\frac{1}{2}\left(k-\alpha j \right)}\|f\|_{2}  & 2k < j \le \frac{k}{\alpha},\\ 
					2^{\left(\frac{1}{2}-\alpha \right)k}\|f\|_{2}  & 4\alpha k < j \le 2k,\\
					2^{\frac{1}{4}\left(2k-j\right)}\|f\|_{2}  & k \le j \le 4\alpha k.
				\end{cases}
			\end{equation}
			for all \( f \) with supp \( \hat{f} \subset\left\{\xi:|\xi|\sim 2^k \right\} \).
		\end{lemma}
		
		\begin{proof}
			As the previous proof, we can decompose \( f \) as $f=\sum_{k=0}^{\infty} f_{k}$. If $k\lesssim 1$, we can get (\ref{klesssim3}) as well. If $k\gg 1$, we can divide $t\in(0,1)$ into three parts as
			(\ref{tdivision1}) to define 
			\[
			\left\|\sup _{0<t<1} \frac{\left|U_\gamma f_k(x,t)-f_k(x)\right|}{t^{\delta}}\right\|_{L^{2}([-1,1])}:=I +I I+ I I I.
			\]
			For the first term $I$, we will treat it as (\ref{firstterm2-1}), then obtain the same estimate as (\ref{firstterm2-2}) and (\ref{firstterm2-3}) to get
			\[
			\left\|\sup _{0<t<2^{-\frac{k}{\alpha}}} \frac{\left|f_{k}(\gamma(x, t))-f_{k}(x)\right|}{t^{ \delta}}\right\|_{L^{2}([-1,1])}\lesssim 2^{\frac{\delta k}{\alpha}}\left\|{f}_{k}\right\|_{L^{2}\left(\mathbb{R}\right)}.
			\]
			As for the third term, we can handle it similar to (\ref{divison3}), that is
			\[
			\left\|\sup _{2^{-k}\leq t<1} \frac{\left|U_\gamma f_k(x,t)-f_k(x)\right|}{t^{\delta}}\right\|_{L^{2}([-1,1])}\lesssim 2^{k\left(\delta+\frac{1}{4}\right)}\|{f}_{k}\|_{L^{2}\left(\mathbb{R}\right)}.
			\]
			As for the second term, we can get
			\begin{align*}
				&\left\|\sup _{2^{-\frac{k}{\alpha}}\leq t< 2^{-k}} \frac{\left|U_\gamma f_k(x,t)-f_k(x)\right|}{t^{\delta}}\right\|_{L^{2}([-1,1])} \\
				&\leq \left\|\sup _{k\leq j \leq \frac{k}{\alpha}} \sup _{2^{-j-1}\leq t< 2^{-j}} \frac{\left|U_\gamma f_k(x,t)\right|}{t^{\delta}}\right\|_{L^{2}([-1,1])}+ 2^{\frac{\delta k}{\alpha}}\|{f}_{k}\|_{L^{2}\left(\mathbb{R}^{d}\right)}\\
				&\lesssim \sum _{k\leq j \leq \frac{k}{\alpha}}2^{j\delta} \left\| \sup _{2^{-j-1}\leq t< 2^{-j}} \left|U_\gamma f_k(x,t)\right|\right\|_{L^{2}([-1,1])}+ 2^{\frac{\delta k}{\alpha}}\|{f}_{k}\|_{L^{2}\left(\mathbb{R}\right)}\\
				&\lesssim  \sum _{k\leq j \leq 4\alpha k}2^{j\delta} \left\| \sup _{t\in (0,2^{-j})} \left|U_\gamma f_k(x,t)\right|\right\|_{L^{2}([-1,1])}+\sum _{ 4\alpha k \leq j \leq 2k }2^{j\delta} \left\| \sup _{t\in (0,2^{-j})} \left|U_\gamma f_k(x,t)\right|\right\|_{L^{2}([-1,1])}\\
				&+ \sum _{2k \leq j \leq \frac{k}{\alpha}}2^{j\delta} \left\| \sup _{t\in (0,2^{-j})} \left|U_\gamma f_k(x,t)\right|\right\|_{L^{2}([-1,1])}+2^{\frac{\delta k}{\alpha}}\|{f}_{k}\|_{L^{2}\left(\mathbb{R}\right)}.
			\end{align*}
			Lemma \ref{proposition4} shows that
			\begin{align*}
				I I&\lesssim \left(\sum _{k\leq j \leq 4\alpha k}2^{j\left(\delta-\frac{1}{4}\right)+\frac{k}{2}}+\sum _{ 4\alpha k \leq j \leq 2k }2^{j\delta+\left(\frac{1}{2}-\alpha\right)k} +\sum _{2k \leq j \leq \frac{k}{\alpha}}2^{j\left(\delta-\frac{\alpha}{2}\right)+\frac{k}{2}}+2^{\frac{\delta k}{\alpha}}\right)\|{f}_{k}\|_{L^{2}\left(\mathbb{R}\right)}\\
				&\lesssim
				\begin{cases}
					\left(k2^{k\left(\delta+\frac{1}{4}\right)}+k2^{k\left(2\delta+\frac{1}{2}-\alpha\right)}+k2^{k\left(2\delta+\frac{1}{2}-\alpha\right)}+2^{k\frac{\delta}{\alpha}}\right)\|{f}_{k}\|_{L^{2}\left(\mathbb{R}\right)}&\delta \in\left[0,\frac{\alpha}{2}\right),\\
					\left(k2^{k\left(\delta+\frac{1}{4}\right)}+k2^{k\left(2\delta+\frac{1}{2}-\alpha\right)}+k2^{k\frac{\delta}{\alpha}}+2^{k\frac{\delta}{\alpha}}\right)\|{f}_{k}\|_{L^{2}\left(\mathbb{R}\right)}&\delta \in\left[\frac{\alpha}{2},\frac{1}{4}\right),\\
					\left(k2^{k\left(4\alpha \delta+\frac{1}{2}-\alpha\right)}+k2^{k\left(2\delta+\frac{1}{2}-\alpha\right)}+k2^{k\frac{\delta}{\alpha}}+2^{k\frac{\delta}{\alpha}}\right)\|{f}_{k}\|_{L^{2}\left(\mathbb{R}\right)}&\delta \in\left[\frac{1}{4},\alpha \right),\\
				\end{cases}\\
				&\lesssim
				\begin{cases}
					k2^{k\left(\delta+\frac{1}{4}\right)}\|{f}_{k}\|_{L^{2}\left(\mathbb{R}\right)}&\delta \in\left[0,\alpha-\frac{1}{4}\right),\\
					k2^{k\left(2\delta+\frac{1}{2}-\alpha\right)}\|{f}_{k}\|_{L^{2}\left(\mathbb{R}\right)}&\delta \in\left[\alpha-\frac{1}{4},\frac{\alpha}{2}\right),\\
					k2^{k\frac{\delta}{\alpha}}\|{f}_{k}\|_{L^{2}\left(\mathbb{R}\right)}&\delta \in\left[\frac{\alpha}{2},\alpha \right).\\
				\end{cases}
			\end{align*}
			Hence, for any $\varepsilon>0 $,  $I+I I+I I I$ can be estimated by
			\[
			\left\|\sup _{0<t<1} \frac{\left|U_\gamma f_k(x,t)-f_k(x)\right|}{t^{\delta}}\right\|_{L^{2}([-1,1])} \lesssim_{\varepsilon}
			\begin{cases}
				2^{k\left(\delta+\frac{1}{4}+\varepsilon\right)}\|{f}_{k}\|_{L^{2}\left(\mathbb{R}\right)}&\delta \in\left[0,\alpha-\frac{1}{4}\right),\\
				2^{k\left(2\delta+\frac{1}{2}-\alpha+\varepsilon\right)}\|{f}_{k}\|_{L^{2}\left(\mathbb{R}\right)}&\delta \in\left[\alpha-\frac{1}{4},\frac{\alpha}{2}\right),\\
				2^{k\left(\frac{\delta}{\alpha}+\varepsilon\right)}\|{f}_{k}\|_{L^{2}\left(\mathbb{R}\right)}&\delta \in\left[\frac{\alpha}{2},\alpha \right).
			\end{cases}
			\]

		\end{proof}

		\section{Necessary conditions}\label{chap4}
		
		In order to prove necessity in {Theorem \ref{theorem1}}, {Theorem \ref{theorem2}}, {Theorem \ref{theorem3}} and {Theorem \ref{theorem4}}, we use arguments from Nikishin-Stein theory, which means we only need to construct counterexamples of maximal functions.

		\begin{proposition}\label{steinmax}
			Let $\delta\in [0,1)$. Assuming that for every \( f \in H^{s}(\mathbb{R}^{d}) \), the limit \( \underset{t \rightarrow 0}{\lim} \frac{U_{\gamma} f(x, t)-f(x)}{t^\delta} \) exists for almost every \( x \in \mathbb{R}^{d} \). Then we have
			\begin{equation}\label{strong}
				\left\|\sup _{t \in (0,1)}\left|\frac{U_{\gamma} f(x, t)-f(x)}{t^\delta}\right|\right\|_{L^{2}(B(0,1))} \lesssim \|f\|_{H^{s}\left(\mathbb{R}^{d}\right)}.
			\end{equation}
		\end{proposition}

		\begin{theorem}\label{counterex1}
			Let $d\geq 1$ and the curve $\gamma(x,t)=x-\left(t^\alpha,0,…,0\right)$ with $\alpha \geq \frac{1}{2}$. For $\delta\in [0,1)$, the corresponding maximal estimate (\ref{strong}) yields $s\geq \frac{d}{2(d+1)}+\delta$.
		\end{theorem}
		\begin{proof}
			This counterexample is inspired by Bourgain \cite{MR3574661}. For the convenience of readers, we may sketch some proof.
			
			We write \( x=\left(x_{1}, \ldots, x_{n}\right)=\left(x_{1}, x^{\prime}\right) \in B(0,1) \subset \mathbb{R}^{d} \). Let \( \varphi: \mathbb{R} \rightarrow \mathbb{R}_{+} \) and \( \Phi: \mathbb{R}^{d-1} \rightarrow \mathbb{R}_{+} \) satisfy supp \( \hat{\varphi} \subset[-1,1] \), supp \( \hat{\Phi} \subset B(0,1), \hat{\varphi}, \hat{\Phi} \) smooth, and \( \varphi(0)=\Phi(0)=1 \). Set \( D=R^{\frac{d+2}{2(d+1)}} \), and define
			\[
			f(x)=e^{iR x_{1}} \varphi\left(R^{1 / 2} x_{1}\right) \Phi\left(x^{\prime}\right) \prod_{j=2}^{d}\left(\sum_{\frac{R}{2 D}<\ell_{j}<\frac{R}{D}} e^{i D \ell_{j} x_{j}}\right),
			\]
			where \( \ell=\left(\ell_{2}, \ldots, \ell_{d}\right) \in \mathbb{Z}^{d-1} \). Then
			\[
			\|f\|_{H^s} \sim R^{s-1 / 4}\left(\frac{R}{D}\right)^{\frac{d-1}{2}}.
			\]
			It is easy to get
			\[
			\begin{aligned}
				U_\gamma f(x)&=\frac{1}{(2\pi)^{d}}\iint \hat{\varphi}(\lambda) \hat{\Phi}\left(\xi^{\prime}\right) \\
				&\quad \times\left\{\sum_{\ell} e^{i\left(\left(R+\lambda R^{1 / 2}\right) \left( x_{1}-t^\alpha\right)+\left(\xi^{\prime}+D \ell\right) \cdot x^{\prime}+\left(R+\lambda R^{1 / 2}\right)^{2} t+\left|\xi^{\prime}+D \ell\right|^{2} t\right)}\right\} d \lambda d \xi^{\prime}.
			\end{aligned}
			\]
			Taking \( |t|<\frac{c}{R},|x|<c \), for suitable constant \( c>0 \), one gets
			\begin{equation}\label{bourgaincex}
				\begin{aligned}
					\left|U_\gamma f(x)\right| & \sim\left|\int \hat{\varphi}(\lambda)\left\{\sum_{\ell} e\left(\lambda R^{1 / 2} x_{1}+D \ell \cdot x^{\prime}+2 \lambda R^{3 / 2} t+D^{2}|\ell|^{2} t\right)\right\} d \lambda\right| \\
					& \sim \varphi\left(R^{1 / 2}\left(x_{1}+2 R t\right)\right)\left|\sum_{\ell} e\left(D \ell \cdot x^{\prime}+D^{2}|\ell|^{2} t\right)\right| .
				\end{aligned}
			\end{equation}
			More precisely, choose \( t=-\frac{x_1}{2R}+\tau \) with \( |\tau|<R^{-3 / 2} / 10 \) to ensure the first factor in (\ref{bourgaincex}) is \( \sim 1 \). Via the quadratic Gauss sum evaluation, for second factor, we have
			\[
			\left|\sum_{\ell} e\left(D \ell \cdot x^{\prime}+D^{2}|\ell|^{2} t\right)\right|\sim R^{\frac{d-1}{4}}
			\]
			on a set \( A_{R} \subset B(0,1) \) with $m\left(A_{R}\right) \gtrsim \frac{1}{log log R}$. At the same time, upper bound for \( \left|f(x)\right| \) from \cite{Fan-Wang} shows that
			\[
			\left|f(x)\right| \lesssim R^{\frac{d-1}{2(d+1)}}, \quad\left(x_{1}, x^{\prime}\right) \in A_{R}.
			\]
		If maximal estimate (\ref{strong})  holds, we have
			\[
			\begin{aligned}
				\frac{\left\|\underset{t \in (0,1)}{\sup} \left|\displaystyle\frac{U_{\gamma} f(x, t)-f(x)}{t^\delta}\right|\right\|_{L^{1}(B(0,1))}}{\|f\|_{H^s}}&\gtrsim \frac{R^{\frac{d-1}{4}}R^{\delta}}{R^{s-1 / 4}\left(\frac{R}{D}\right)^{\frac{d-1}{2}}}\\
				&=R^{\delta+\frac{d}{2(d+1)}-s}.
			\end{aligned}
			\]
			Let $R \to \infty$, we can get $s\geq \frac{d}{2(d+1)}+\delta$.
		\end{proof}

		\begin{theorem}\label{counterex2}
			Let $d\geq 1$ and the  curve $\gamma(x,t)=x-t^\alpha e_1$ with $\alpha \geq \frac{1}{2}$.  For $\delta\in [0,\alpha)$, the corresponding maximal estimate (\ref{strong}) yields $s\geq 2\delta$.
		\end{theorem}
		\begin{proof}
			
			We write \( x=\left(x_{1}, \ldots, x_{d}\right)=\left(x_{1}, x^{\prime}\right) \in B(0,1) \subset \mathbb{R}^{d} \). Let  \( g \in  C_0^{\infty} \) with supp $g\subset [-\frac{1}{2},\frac{1}{2}]$ such that \( g(\xi) \geq 0 \) and \( \int g(\xi) d \xi=1 \). We construct a family of functions \( f_{ R,\varepsilon} \), with large \(R \), defined by
			\[
			\widehat{f}_{ R,\varepsilon}(\eta)=\frac{1}{R} g\left( \frac{\eta_1+R^{1+\varepsilon}}{R}\right)g(\eta_2)…g(\eta_n).
			\]
			It is easy to calculate the Sobolev norm of \( {f}_{ R,\varepsilon} \)
			\begin{equation}\label{sobolevnorm}
				\left\|f_{R,\varepsilon}\right\|_{H^{s}} \lesssim R^{(1+\varepsilon)s-\frac{1}{2}}.
			\end{equation}
			After changing variables $\frac{\eta_1+R^{1+\varepsilon}}{R}\to \xi$, we have
			\[
			\begin{aligned}
				\left|U_\gamma f_{ R,\varepsilon}\left(x, t\right)\right|&= \left|\int  \frac{ e^{i\left((x_1-t^\alpha )\eta_1+t|\eta_1|^{2}\right)}   }{R} g\left( \frac{\eta_1+R^{1+\varepsilon}}{R}\right)\frac{d \eta_1}{2 \pi}\right|  \prod_{j=2}^{d}   \left|\int   e^{i\left(x_j \eta_j+t|\eta_j|^{2}\right)}  g\left( \eta_j \right)\frac{d \eta_j}{2 \pi}\right|\\
				&= \left|\int e^{i \Phi_{ R,\varepsilon}\left(\xi ; x_1, t\right)} g(\xi) \frac{d \xi}{2 \pi}\right|  \prod_{j=2}^{d}   \left|\int   e^{i\left(x_j \eta_j+t|\eta_j|^{2}\right)}  g\left( \eta_j \right)\frac{d \eta_j}{2 \pi}\right|
			\end{aligned}
			\]
			where $\Phi_{ R,\varepsilon}\left(\xi ; x_1, t\right)=(x_1-2R^{1+\varepsilon}t) R \xi+t  R^2 \xi^2-t^\alpha R\xi$ . Let 
			\[
			t_{x_1}=\frac{x_1}{R^{1+\varepsilon}},
			\]
			and for a suitably small constant $c$, consider $x\in $
			\[
			A_x=\left\{x:\left( \frac{c}{2} R^{-1+\varepsilon}<x_1<cR^{-1+\varepsilon}\right)\times (0,c)^{d-1}\right\}.
			\]
			We can find $t_{x_1}\in \left( \frac{c}{2} R^{-2},cR^{-2}\right)$, which means $\Phi_{ R}\left(\xi ; x_1, t_{x_1}\right)\leq \frac{1}{100}$. So we can get
			\[
			\begin{aligned}
				\left|\int e^{i \Phi_{ R,\varepsilon}\left(\xi ; x_1, t_{x_1}\right)} g(\xi) \frac{d \xi}{2 \pi}\right| & \geq \int g(\xi) \frac{d \xi}{2 \pi}-\int\left|e^{i \Phi_{R,\varepsilon}\left(\xi ; x_1, t_{x_1}\right)}-1\right| g(\xi) \frac{d \xi}{2 \pi} \\
				& \geq \frac{1}{2\pi}-\max _{|\xi| \leq 1 / 2}\left|e^{i \Phi_{R,\varepsilon}\left(\xi ; x_1, t_{x_1}\right)}-1\right|\\
				& \geq \frac{1}{4\pi}.
			\end{aligned}
			\]
			In addition, notice that $x_{i}<c$ and $t_{x_1}\le \frac{c}{R^2}$ for $i=2 ,…, d $,  we can also get
			\[
			\left|\int   e^{i\left(x_i \eta_i+t_{x_1}|\eta_i|^{2}\right)}  g\left( \eta_i \right)\frac{d \eta_i}{2 \pi}\right|\geq \frac{1}{4\pi},
			\]
			for $i=2 ,…, d $. Thus we have
			\begin{equation}\label{fanli1-maxestchai1}
					\sup _{t}\left|U_\gamma f_{ R,\varepsilon}\left(x, t\right)\right| \geq\left|U_\gamma f_{R,\varepsilon}\left(x,t_{x_1} \right)\right| \geq\left(\frac{1}{4\pi}\right)^{d}.
			\end{equation}
			Notice that $x_1 R \sim R^{\varepsilon}\gg 1$ when $R\to \infty$, which means
			\[
			\begin{aligned}
			 \left|f_{ R,\varepsilon}(x)\right|&=\left(\frac{1}{2\pi}\right)^{d}\left| \int_\mathbb{R} e^{ix_1 \eta_1} \frac{1}{R} g\left( \frac{\eta_1+R^{1+\varepsilon}}{R}\right)d\eta_1 \right|    \prod_{i=2}^{d}  \left| \int_\mathbb{R} e^{ix_i \eta_i}  g\left( \eta_i\right)d\eta_i \right| \\
				&\leq \left(\frac{1}{2\pi}\right)^{d}\left| \int_\mathbb{R} e^{ix_1R \xi } g\left( \xi \right)d\xi \right| \le \left(\frac{1}{8\pi}\right)^{d}.
			\end{aligned}
			\]
			So 
			\begin{equation}
				\begin{aligned}
					\frac{\left\|\underset{t \in (0,1)}{\sup}\left|\displaystyle\frac{U_{\gamma} f_{R,\varepsilon}(x, t)-f_{R,\varepsilon}(x)}{t^\delta}\right|\right\|_{L^{2}(B(0,1))}}{\left\|f_{R,\varepsilon}\right\|_{H^{s}}}&\gtrsim \frac{\left\| \left|\displaystyle\frac{U_{\gamma} f_{R,\varepsilon}(x, t_x)-f_{R,\varepsilon}(x)}{t_{x_1}^\delta}\right|\right\|_{L^{2}(B(0,1))}}{R^{(1+\varepsilon)s-\frac{1}{2}}}\\
					&\gtrsim \frac{R^{2\delta}R^{\frac{-1+\varepsilon}{2}}}{R^{(1+\varepsilon)s-\frac{1}{2}}}.
				\end{aligned}
			\end{equation}
			Let $R \to \infty$, we can get $s\geq \frac{2\delta+\frac{\varepsilon}{2}}{1+\varepsilon}$. Let $\varepsilon\to 0$, we can get $s\geq 2\delta$.
			
		\end{proof}

		\begin{theorem}\label{counterex3}
			Let  $d= 1$ and the curve $\gamma(x,t)=x-t^\alpha$ with $0<\alpha<\frac{1}{2}$. For $\delta\in [0,\alpha)$, the corresponding maximal estimate (\ref{strong}) yields $s\geq \frac{1}{2}-\alpha+2\delta$.
			\begin{proof}
				Let  \( g \in  C_0^{\infty} \) with supp $g\subset [-\frac{1}{2},\frac{1}{2}]$ such that \( g(\xi) \geq 0 \) and \( \int g(\xi) d \xi=1 \). We construct a family of functions \( f_{ R} \), with large \( R  \), defined by
				\[
				\widehat{f}_{R}(\eta)=\frac{1}{R} g\left( \frac{\eta}{R}\right)
				\]
				with $
					\left\|f_{R}\right\|_{H^{s}} \lesssim R^{s-\frac{1}{2}}
			$.
				After scaling, we have
				\[
				\left|U_\gamma f_{ R}\left(x, t\right)\right|=\left|\int e^{i\left((x-t^\alpha) \eta+t|\eta|^{2}\right)} \frac{1}{R} g(\frac{\eta}{R})  \frac{d \eta}{2 \pi}\right|=\left|\int e^{i \Phi_{ R}\left(\xi ; x, t\right)} g(\xi) \frac{d \xi}{2 \pi}\right|,
				\]
				where $\Phi_{ R}\left(\xi ; x, t\right)=(x-t^\alpha) R \xi+t  R^2 \xi^2$. Let 
				\[
				t_x=x^{\frac{1}{\alpha}},
				\]
				and for a suitably small constant $c$, consider $x$ in
				\[
				A_x=\left\{ x: \frac{c}{2} R^{-2\alpha}<x<cR^{-2\alpha}\right\}.
				\]
				We can find $t_x\in \left( \frac{c}{2} R^{-2},cR^{-2}\right)$, which means $\Phi_{ R}\left(\xi ; x, t\right)\leq \frac{1}{100}$. So we can get
				\begin{equation}\label{fanli1-maxestchai}
						\sup _{t}\left|U_\gamma f_{ R}\left(x, t\right)\right| \geq\left|U_\gamma f_{\lambda_{j}, R_{j}}\left(x, t_{x}\right)\right|
						 \geq \frac{1}{4\pi}.
				\end{equation}
				Notice that $xR \sim R^{1-2\alpha}\gg 1$ when $R\to \infty$, which means
				\[
				\left|f_{ R}(x)\right|=\frac{1}{2\pi}\left| \int_\mathbb{R} e^{ix\eta } \frac{1}{R} g\left( \frac{\eta}{R}\right)d\eta \right|=\frac{1}{2\pi}\left| \int_\mathbb{R} e^{ixR \xi } g\left( \xi \right)d\eta \right|\le \frac{1}{8\pi}.
				\]
				So the maximal estimate (\ref{strong}) yields
				\[
					1 \gtrsim \frac{R^{2\delta-\alpha}}{R^{s-\frac{1}{2}}}, 
				\]
				Let $R \to \infty$, we can get $s\geq \frac{1}{2}-\alpha+2\delta$.
			\end{proof}
			
		\end{theorem}

		\begin{theorem}\label{counterex4}
			Let  $d= 1$ and the curve $\gamma(x,t)=x-t^\alpha$ with $\frac{1}{4}\leq \alpha \leq 1$. For $\delta\in [0,\alpha)$, the corresponding maximal estimate (\ref{strong}) yields $s\geq \frac{1}{4}+\delta$.
		\end{theorem}

		\begin{proof}
			Let  \( g \in  C_0^{\infty} \) with supp $g\subset [-\frac{1}{2},\frac{1}{2}]$ such that \( g(\xi) \geq 0 \) and \( \int g(\xi) d \xi=1 \). For  large \( R  \), define
			\[
			\widehat{f}_{R}(\eta)=\frac{1}{R} g\left( \frac{\eta+R^2}{R}\right).
			\]
			with
				$\left\|f_{R}\right\|_{H^{s}} \lesssim R^{2s-\frac{1}{2}}$.
			Changing variables $\frac{\eta+R^2}{R} \to \xi$, we have
			\[
			\left|U_\gamma f_{ R}\left(x, t\right)\right|=\left|\int e^{i\left((x-t^\alpha) \eta+t|\eta|^{2}\right)} \frac{1}{R} g(\frac{\eta+R^2}{R})  \frac{d \eta}{2 \pi}\right|=\left|\int e^{i \Phi_{ R}\left(\xi ; x, t\right)} g(\xi) \frac{d \xi}{2 \pi}\right|,
			\]
			where $\Phi_{ R}\left(\xi ; x, t\right)=(x-t^\alpha-2R^2 t) R \xi+t  R^2 \xi^2$.
			For $x\in \left( \frac{c}{2},c\right)$ with   a suitably small constant $c$, we obtain a unique solution $t_x\in (0,cR^{-2})$ satisfying
			\[
			x-t_x^\alpha-2R^2 t_x =0.
			\]
			Thus we can get  $\Phi_{ R}\left(\xi ; x, t\right)\leq \frac{1}{100}$. So we can get the estimate (\ref{fanli1-maxestchai}) as well.  In addition $xR \sim R\gg 1$ when $R\to \infty$, which means
			\[
			\left|f_{ R}(x)\right|=\frac{1}{2\pi}\left| \int_\mathbb{R} e^{ix\eta } \frac{1}{R} g\left( \frac{\eta+R^2}{R}\right)d\eta \right|=\frac{1}{2\pi}\left| \int_\mathbb{R} e^{ixR \xi } g\left( \xi \right)d\xi \right|\le \frac{1}{8\pi}.
			\]
			So the maximal estimate (\ref{strong}) yields
			\[
			1\gtrsim \frac{R^{2\delta}}{R^{2s-\frac{1}{2}}}.
			\]
			Let $R \to \infty$, we can get $s<\frac{1}{4}+\delta$.
		\end{proof}

		\begin{theorem}\label{counterex5}(inspired by \cite{Li2021OnCP})
			Let  $d= 1$ and the curve $\gamma(x,t)=x+t^\alpha$ with $0<\alpha<\frac{1}{2}$.  For $\delta\in [0,\alpha)$, the corresponding maximal estimate (\ref{strong}) yields $s\geq \frac{\delta}{\alpha}$.
		\end{theorem}
		\begin{proof}
			For \( R \geq 1 \), we define
			$
			\widehat{f_{R}}(\xi)=\chi_{\left\{\xi: R \leq \xi\leq R+1\right\}}(\xi)
			$
			with $\left\|f_{R}\right\|_{H^{s}} \lesssim R^{s}$. By Taylor's formula, for arbitrary \( t \in(0,1) \), we get
			\[
			\begin{aligned}
				&\left|U_\gamma f_R (x,t)-f_{R}(x)\right|\\				&=\frac{1}{2\pi }\left|\int_{\left\{\xi: R \leq \xi\leq R+1\right\}} e^{i x \cdot \xi}\left(i {t^{\alpha}} \xi+i t \xi^{2}\right) d \xi+\sum_{j \geq 2} \frac{1}{j !} \int_{\left\{\xi: R \leq \xi\leq R+1\right\}} e^{i x \cdot \xi}\left(i {t^{\alpha}}\xi+i t \xi^{2}\right)^{j} d \xi\right| \\
				&\geq\frac{1}{2\pi } \left| \left|\int_{\left\{\xi: R \leq \xi\leq R+1\right\}} e^{i x \cdot \xi}\left(i {t^{\alpha}} \xi+i t \xi^{2}\right) d \xi\right|-\left| \sum_{j \geq 2} \frac{1}{j !} \int_{\left\{\xi: R \leq \xi\leq R+1\right\}} e^{i x \cdot \xi}\left(i {t^{\alpha}} \xi+i t \xi^{2}\right)^{j} d \xi\right|\right| .
			\end{aligned}
			\]
			We choose \( t_{0}=cR^{-\frac{1} { \alpha}} \). Here, $c$ is a sufficiently small constant. For each \( \xi \in \left\{\xi: R \leq \xi\leq R+1\right\} \), with large $R$, we have
			\[
			c^{\alpha} \leq {t_{0}^{\alpha}}\xi+t_{0} \xi^{2} \leq c^{\alpha} \frac{(R+1)}{R}+c\frac{(R+1)^{2}}{R^{1 / \alpha}}\leq 2c^{\alpha}.
			\]
			Therefore, we have
			\[
				\left|\int_{\left\{\xi: R \leq \xi\leq R+1\right\}} e^{i x  \xi}\left(i {t_{0}^{\alpha}}\xi+i t_{0} \xi^{2}\right) d \xi\right| \geq \frac{c^{\alpha}}{2}
			\]
			for any \( x \in B\left(0, c\right) \). In addition
			\[
			\begin{aligned}
				&\left|\sum_{j \geq 2} \frac{1}{j !} \int_{\left\{\xi: R \leq \xi\leq R+1\right\}} e^{i x \cdot \xi}\left(i {t_{0}^{\alpha}} \xi+i t_{0} \xi^{2}\right)^{j} d \xi\right| \\
				&\leq \sum_{j \geq 2} \frac{1}{j !} \int_{\left\{\xi: R \leq \xi\leq R+1\right\}}\left|{t_{0}^{\alpha}} \xi+t_{0} \xi^{2}\right|^{j} d \xi \\
				&\leq \sum_{j \geq 2} \frac{1}{j !}\left(2c^{\alpha}\right)^{j} \\
				&\leq \frac{c^{\alpha}}{4}
			\end{aligned}
			\]
			for each \( x \in B\left(0, c\right) \). Therefore, we have
			\[
			1\gtrsim R^{\frac{\delta}{\alpha}-s},
			\]
			Let $R \to \infty$, we can get $s\geq \frac{\delta}{\alpha}$.
		\end{proof}

		If we consider Schwartz function, it is easy to get the convergence rate $\delta$ cannot exceed $\alpha$. Since the proof is easy to get, we omit it.
		\begin{theorem}\label{counterex6}
			For the curve $\gamma(x, t)=x-e_{1} t^{\alpha} $ with \( 0<\alpha\leq 1 \)
			and for each Schwartz function \( f \), if
			\begin{equation}\label{schw}
				\lim _{t \rightarrow 0^{+}} \frac{U_\gamma f(x, t)-f(x)}{t^{\alpha}}=0 \text { a.e. } x \in \mathbb{R}^{d},
			\end{equation}
			then \( f \equiv 0 \).
		\end{theorem}

		For the convenience of readers, the following table presents the counterexamples required to prove the sharpness of {Theorem \ref{theorem1}}, {Theorem \ref{theorem2}}, {Theorem \ref{theorem3}} and {Theorem \ref{theorem4}}.
		\begin{table}[ht]
			\begin{tabular}{m{4cm}<{\centering}|m{10cm}<{\centering}} 
				\hline   \textbf{Theorem} &\textbf{Counterexamples required} \\     
				\hline
				\footnotesize
				{Theorem \ref{theorem1}} &{Theorem \ref{counterex1}, Theorem \ref{counterex2}, Theorem \ref{counterex6}}  \\
				
				\hline
				\footnotesize
				{Theorem \ref{theorem2}} &{Theorem \ref{counterex4}, Theorem \ref{counterex2}, Theorem \ref{counterex6}}  \\
				
				\hline
				\footnotesize
				{Theorem \ref{theorem3}} &{Theorem \ref{counterex3}, Theorem \ref{counterex5}, Theorem \ref{counterex6}}  \\
				
				\hline
				\footnotesize
				{Theorem \ref{theorem4}} &{Theorem \ref{counterex4}, Theorem \ref{counterex3}, Theorem \ref{counterex5}, Theorem \ref{counterex6}}  \\
				
				\hline
			\end{tabular}   
		\end{table}
		
		\section{Convergence rate of fractional Schrödinger operator along curves}\label{chap5}

		The previous method can also be applied to handle the convergence rate of fractional Schrödinger operator along curves in $d=1$. Here, we will only provide results.

		\begin{theorem}\label{theorem5}
			Let $\delta \in[0,\alpha)$ and $m\in(0,1)$. For any curve $\gamma(x,t): \mathbb{R} \times[-1,1] \rightarrow \mathbb{R}$ satisfying $\gamma(x,0)=x$ which is bilipschitz continuous in \( x \), and $\alpha$-Hölder continuous in \( t \) with \(  \frac{1}{2}< \alpha \leq 1\), we have
			\begin{equation}\label{pt convergence in r1}
				\lim _{t \rightarrow 0} \frac{U_{\gamma}^m f(x, t)-f(x)}{t^\delta}=0 \quad \text { a.e. } x \in \mathbb{R},\quad \forall f \in H^{s}\left(\mathbb{R}\right)
			\end{equation}
			whenever 
			\[
			s>s(\delta)=\begin{cases}
				\frac{2-m}{4}   \quad  &0  \leq\delta\leq \frac{2\alpha-1}{4}, \\ 
				m\delta+\frac{1-m\alpha}{2}   \quad   &  \frac{2\alpha-1}{4}<\delta\leq \frac{\alpha}{2}.  \\
				\frac{\delta}{\alpha} \quad   &  \delta>\frac{\alpha}{2}.  \\
			\end{cases}
			\]
		\end{theorem}
		
		\begin{theorem}\label{theorem6}
			Let $\delta,m$ satisfy the assumptions in Theorem \ref{theorem5}.  For any curve $\gamma(x,t): \mathbb{R} \times[-1,1] \rightarrow \mathbb{R}$ satisfying $\gamma(x,0)=x$ which is bilipschitz continuous in \( x \), and $\alpha$-Hölder continuous in \( t \) with \(  0< \alpha \leq \frac{1}{2}\), we have (\ref{pt convergence in r1}) whenever
			\[
			s>s(\delta)=\begin{cases}
				m\delta+\frac{1-m\alpha}{2}    \quad  &0  \leq \delta\leq \frac{\alpha}{2}, \\ 
				\frac{\delta}{\alpha} \quad   &  \delta>\frac{\alpha}{2}.  \\
			\end{cases}
			\]. 
		\end{theorem}
		
		\begin{theorem}\label{theorem7}
			Let $\delta \in[0,\alpha)$ and $m\in(1,\infty )$. For any curve $\gamma(x,t): \mathbb{R} \times[-1,1] \rightarrow \mathbb{R}$ satisfying $\gamma(x,0)=x$ which is bilipschitz continuous in \( x \), and $\alpha$-Hölder continuous in \( t \) with \(  \frac{1}{m} \leq \alpha \leq 1  \), we have (\ref{pt convergence in r1}) whenever  
			\[
			s>s(\delta)=\begin{cases}
				(m-1)\delta+\frac{1}{4}    \quad  &0  \leq \delta\leq \frac{1}{4}, \\ 
				m\delta \quad   &  \delta>\frac{1}{4}.  \\
			\end{cases}
			\]. 
		\end{theorem}
		
		\begin{theorem}\label{theorem8}
			Let $\delta,m$ satisfy the assumptions in Theorem \ref{theorem7}.  For any curve $\gamma(x,t): \mathbb{R} \times[-1,1] \rightarrow \mathbb{R}$ satisfying $\gamma(x,0)=x$ which is bilipschitz continuous in \( x \), and $\alpha$-Hölder continuous in \( t \) with \( 0< \alpha \leq  \frac{1}{2m}   \), we have (\ref{pt convergence in r1}) whenever 
			\[
			s>s(\delta)=\begin{cases}
				m\delta+\frac{1-m\alpha}{2}    \quad  &0  \leq \delta\leq \frac{\alpha}{2}, \\ 
				\frac{\delta}{\delta} \quad   &  \delta>\frac{\alpha}{2}.  \\
			\end{cases}
			\]. 
		\end{theorem}
		
		\begin{theorem}\label{theorem9}
			Let $\delta,m$ satisfy the assumptions in Theorem \ref{theorem7}.  For any curve $\gamma(x,t): \mathbb{R} \times[-1,1] \rightarrow \mathbb{R}$ satisfying $\gamma(x,0)=x$ which is bilipschitz continuous in \( x \), and $\alpha$-Hölder continuous in \( t \) with \(  \frac{1}{2m} < \alpha <\min\left\{ \frac{1}{2},\frac{1}{m} \right\} \), we  have (\ref{pt convergence in r1}) whenever\[
			s>s(\delta)=\begin{cases}
				(m-1)\delta+\frac{1}{4}    \quad  &0  \leq \delta\leq \frac{2m\alpha-1}{4}, \\ 
				m\delta+\frac{1-m\alpha}{2}   \quad  &\frac{2m\alpha-1}{4}  <\delta\leq \frac{\alpha}{2}, \\ 
				\frac{\delta}{\alpha} \quad   &  \delta>\frac{\alpha}{2}.  \\
			\end{cases}
			\]. 
		\end{theorem}

		\begin{theorem}\label{theorem10}
			Let $\delta \in[0,\alpha)$ and $m\in(1,2 )$. For any curve $\gamma(x,t): \mathbb{R} \times[-1,1] \rightarrow \mathbb{R}$ satisfying $\gamma(x,0)=x$ which is bilipschitz continuous in \( x \), and $\alpha$-Hölder continuous in \( t \) with \(  \frac{1}{2} \leq \alpha < \frac{1}{m}  \), we have (\ref{pt convergence in r1}) whenever
			\[
			s>s(\delta)=\begin{cases}
				(m-1)\delta+\frac{1}{4}    \quad  &0  \leq \delta\leq \frac{1}{4}, \\ 
				\frac{m-1}{1-\alpha}\delta+\frac{1-m\alpha}{4(1-\alpha)}    \quad  & \frac{1}{4}  <\delta\leq \frac{2\alpha-1}{4}, \\ 
				m\delta+\frac{1-m\alpha}{2}   \quad  &\frac{2\alpha-1}{4}  <\delta\leq \frac{\alpha}{2}, \\ 
				\frac{\delta}{\alpha} \quad   &  \delta>\frac{\alpha}{2}.  \\
			\end{cases}
			\]. 
		\end{theorem}

		Meng Wang, Department of Mathematics, Zhejiang University, Hangzhou 310058,  People’s Republic of China
		
		E-mail address: mathdreamcn@zju.edu.cn
		
		Zhichao Wang, Department of Mathematics, Zhejiang University, Hangzhou 310058,  People’s Republic of China
		
		E-mail address: zhichaowang@zju.edu.cn

	\end{sloppypar}
\end{document}